\documentclass[a4paper,10pt]{amsart}
\usepackage{amssymb,amsmath,amsfonts,amsthm}
\usepackage[dvips]{graphicx}
\usepackage{psfrag}
\usepackage[usenames, dvipsnames]{color}

\subjclass{ 37C15; 37C40; 37D20; 37D30}
\keywords{Lyapunov exponents, rigidity, Anosov diffeomorphim, partially hyperbolic diffeomorphism}

\theoremstyle{plain}
\newtheorem{main}{Theorem}

\newtheorem{maincor}[main]{Corollary}
\newtheorem{theorem}{Theorem}[section]
\newtheorem{lemma}[theorem]{Lemma}
\newtheorem{proposition}[theorem]{Proposition}

\theoremstyle{remark}
\newtheorem{remark}[theorem]{Remark}
\newtheorem{definition}[theorem]{Definition}

\newtheorem{conjecture}[theorem]{Conjecture}

\newcommand{\Leb}{\operatorname{vol}}

\newcommand{\Gibb}{\operatorname{Gibb}}

           \def\ea{\end{array}}
          \def\ec{\end{center}}
     \def\ed{\end{description}}
        \def\ee{\end{equation}}
       \def\eea{\end{eqnarray}}
     \def\eeaa{\end{eqnarray*}}
 \def\et{\end{thebibliography}}
\def\bib{\bibitem}

\def\Diff{{\rm Diff}}

\def\supp{\operatorname{supp}}

\def\cA{{\mathcal A}}

\def\cL{{\mathcal L}}

\def\cU{{\mathcal U}}

\def\cH{{\mathcal H}}

\def\cF{{\mathcal F}}

\def\cW{{\mathcal W}}

\def\vep{\varepsilon}

\def\TT{{\mathbb T}}

\def\ZZ{{\mathbb Z}}

\title[Lyapunov exponents and rigidity]{Lyapunov exponents and rigidity of Anosov automorphisms and skew products}

\author{Radu Saghin and Jiagang Yang}

\date{\today}

\thanks{R.S. was partially supported by Fondecyt 1171477,  Math AM-Sud project 16-math-06 Physeco and Conicyt PAI80160049, J.Y. was partially supported by CNPq, FAPERJ, CONICYT PAI80160049, and would like to thank Instituto de Matem\'aticas, PUCV for the hospitality.}

\address{Instituto de Matem\'atica, Pontificia Universidad Cat\'olica de Valpara\'iso, Blanco Viel 596, Cerro Bar\'on, Valpara\'iso, Chile}
\email{rsaghin\@@gmail.com}

\address{Department of Mathematics, Southern University of Science and Technology of China, 1088 Xueyuan Rd., Xili, Nanshan District, Shenzhen, Guangdong, China 518055}
\address{Departamento de Geometria, Instituto de Matem\'atica e Estat\'istica, Universidade Federal Fluminense, Niter\'oi, Brazil}
\email{yangjg\@@impa.br}

\begin{document}

\begin{abstract}

In this paper we obtain local rigidity results for linear Anosov diffeomorphisms in terms of Lyapunov exponents. More specifically, we show that given an irreducible linear hyperbolic automorphism $L$ with simple real eigenvalues with distinct absolute values, any small perturbation preserving the volume and with the same Lyapunov exponents is smoothly conjugate to $L$.

We also obtain rigidity results for skew products over Anosov diffeomorphisms. Given a volume preserving partially hyperbolic skew product diffeomorphism $f_0$ over an Anosov automorphism of the 2-torus, we show that for any volume preserving perturbation $f$ of $f_0$ with the same average stable and unstable Lyapunov exponents, the center foliation is smooth.

\end{abstract}

\maketitle

\setcounter{tocdepth}{1} \tableofcontents

\section{Introduction}

This paper deals with the rigidity for some special classes of dynamical systems, namely hyperbolic and partially hyperbolic diffeomorphisms.

\subsection{Anosov diffeomorphisms}

It is well known that hyperbolic diffeomorphism and flows are structurally stable, meaning that small $C^1$ perturbations are topologically conjugate (or equivalent) to the initial system. However the conjugacy is in general only H\H older continuous, and is not absolutely continuous. Having better regularity of the conjugacy is a rare phenomenon and it is called (smooth) rigidity. It is important understanding when rigidity occurs and finding characterizations for this.

There are some necessary conditions for having smooth conjugacy. If the conjugacy between two Anosov diffeomorphisms is $C^1$, then any type of smooth invariant or structure must be preserved by the conjugacy.

Probably the best known example is the {\it periodic data}, the eigenvalues of the return map at the periodic points. There exists a considerable amount of work regarding the converse implication: if the periodic data of two Anosov diffeomorphisms is the same, then the diffeomorphisms are smoothly conjugate. The result is known to be true in dimension 2 by the pioneering work of de la LLave, Marco and Moriy\'on \cite{MM1,L87,MM2,LM}, based on \cite{LMM}. In its highest generality it is still open in higher dimensions, however there are many partial results in this direction.

There is also a counterexample in higher dimension by de la Llave \cite{L92}, showing that in higher dimension one has to add some other hypothesis like irreducibility of the linear map, real simple spectrum, quasi-conformality of the derivative on higher dimensional invariant sub-bundles, etc. The study of the local rigidity of linear Anosov maps, i.e. the smooth conjugacy with the linear Anosov map given the appropriate hypothesis, was more successful. There are many contributions in this direction by de la Llave, Gogolev, Kalinin, Sadovskaya, and others (\cite{MM1,L87,MM2,LM,L92,G08,G,GG,GKS,KS07,KS09,KS10}).

In general there is a parallel between the rigidity results for hyperbolic systems and the rigidity results in Hyperbolic Geometry, going of course through hyperbolic flows. The periodic data has an equivalent in Hyperbolic Geometry, it is the {\it marked length spectrum}, the length of the closed geodesics inside each homotopy class. Two isometric manifolds have the same marked length spectrum, and a conjecture of Burns-Katok \cite{BK} states that the converse is also true. The result is known in dimension 2 by work of Otal \cite{Ot} and Croke \cite{Cr}, and in any dimension with some mild regularity conditions by the recent results of Guillarmou-Lefeuvre \cite{GL}.

A special problem in the rigidity of Anosov systems is finding sufficient conditions for the existence of a smooth conjugacy of an Anosov system with the linear model. On one hand this is a very interesting question, and on another hand this problem may be easier to study, because of the algebraic structure of the linear model.

One consequence of the smooth conjugacy of a hyperbolic system with the linear one is the regularity of the stable and unstable foliations. Again, it is conjectured that the converse implication holds, at least under some suitable conditions. A remarkable result for Anosov diffeomorphisms on surfaces (and Anosov flows in dimension 3) was obtained by Hurder and Katok \cite{HK}, building on previous work of Anosov, Ghys and others: they showed that enough regularity of the stable and unstable foliations (in their case $C^1$ plus little Zygmund) implies local rigidity. There are various other partial results in this direction, in particular for (geodesic) flows, or under the hypothesis of the preservation of some other structures (for example symplectic).

We are interested in another consequence of the smooth conjugacy between an Anosov diffeomorphisms $f$ and its linear part $L$. The pull back of the Lebesgue measure gives a volume on the manifold which is invariant for the given diffeomorphism $f$. Furthermore the Lyapunov exponents of $f$ with respect to the invariant volume must be equal to the Lyapunov exponents of the linear map $L$.

Again it is natural to ask if the converse implication holds: given an Anosov diffeomorphism $f$, possibly with the additional assumption that $f$ is $C^1$ close to its linear part $L$, which preserve a volume and has the Lyapunov exponents equal to the exponents of $L$, is it true that $f$ is smoothly conjugated to $L$? It was known for a while that in dimension 2 the result holds, but there are no known results for Anosov diffeomorphisms on higher dimensional manifolds. We will see that we can give an affirmative answer in many situations.

Let us remark here that obtaining the smooth conjugacy from equal Lyapunov exponents is only possible if one of the maps is linear. This is in contrast with the periodic data, where one can hope to get general rigidity, eventually under some extra hypothesis. The improvement in our case is that we have only a {\it finite set of data}: the Lyapunov exponents, which consist of a finite set of real numbers. One can hope that this data could determine the smooth conjugacy class of the linear Anosov map, but it is definitely not enough in order to determine other smooth conjugacy classes (in fact the set of conjugacy classes is 'infinite dimensional', because different periodic data gives different conjugacy classes).

It is easy to construct examples of two volume preserving maps with the same Lyapunov exponents (different than the exponents of the linear map) and which are not smoothly conjugated, even in dimension two. For example one can modify a linear Anosov map in two ways, by mixing locally the two directions, but in different places of the manifold; this will create two Anosov maps with the same exponents but with different eigenvalues at some periodic points, so they will not be smoothly conjugated.

We first consider 3 dimensional Anosov diffeomorphisms which are also partially hyperbolic. A diffeomorphism $f$ on the Riemannian manifold $M$ is \emph{Anosov} if there exists a $Df$ invariant splitting of the tangent bundle $TM=E^s\oplus E^u$, and a real number $\lambda>0$, such that for all $x\in M$, for all unit vectors $v^i\in E^i_x\setminus \{0\}$ ($i\in\{s,u\}$), and for some suitable Riemannian metric on $M$, we have
$$
\|Df\mid_{E^s_x}(v^s)\|\leq e^{-\lambda},
$$
$$
e^{\lambda}\leq \|Df\mid_{E^u_x}(v^u)\|.
$$
A diffeomorphism $f$ is \emph{partially hyperbolic} if there exists a $Df$-invariant splitting of the tangent bundle $TM = E^{ss}\oplus E^c \oplus E^{uu}$, such that for all $x\in M$, for all unit vectors $v^i\in E^i_x\setminus \{0\}$ ($i\in\{ ss, c, uu\}$), and for some suitable Riemannian metric on $M$, we have
$$
e^{\lambda_1(x)}\leq \|Df\mid_{E^s_x}(v^s)\|\leq e^{\lambda_2(x)},
$$
$$
e^{\lambda_3(x)}\leq \|Df\mid_{E^c_x}(v^c)\|\leq e^{\lambda_4(x)},
$$
$$
e^{\lambda_5(x)}\leq \|Df\mid_{E^u_x}(v^u)\|\leq e^{\lambda_6(x)},
$$
where $\lambda_1(x)\leq \lambda_2(x)<\lambda_3(x)\leq\lambda_4(x)<\lambda_5(x)\leq\lambda_6(x)$, with $\lambda_2(x)<0$ and $\lambda_5(x)>0$. If the functions $\lambda_i(\cdot)$ $i\in\{1,2,\cdots, 6\}$ can be taken to be constant, then we say that $f$ is \emph{absolute partially hyperbolic.}

When the manifold is the three dimensional torus, we can make use of the topological classification of the partially hyperbolic diffeomorphisms over $\TT^3$ developed in a series of papers \cite{BBI09, Ham13,HaP,Pot15,Ure12}. In this case we can show that having the right Lyapunov exponents implies local rigidity for a large class of Anosov diffeomorphisms.

\begin{main}\label{main.3d}

Let $L$ be a three dimensional Anosov automorphism on $\TT^3$ with simple real eigenvalues with distinct absolute values. Suppose $f$ is a $C^\infty$ volume preserving Anosov diffeomorphism which is isotopic to $L$ and is partially hyperbolic. If $f$ has the same Lyapunov exponents as $L$, then $f$ is $C^\infty$ conjugate to $L$.

\end{main}

\begin{remark}\label{rk.analytic}

One can also get results on the $C^k$ differentiability of the conjugacy for any values of $k>1$, or the analycity of the conjugacy, for more about this see \cite{G}, \cite{L97}. 

\end{remark}

A result similar to Theorem \ref{main.3d} was obtained independently and announced recently by Micena-Tahzibi in \cite{MT18}. The assumption in \cite{MT18} is that $f$ is a $C^r$ ($r\geq2$) volume preserving Anosov diffeomorphism isotopic to $L$, also partially hyperbolic, the center (or weak unstable) foliation is absolutely continuous and the center Lyapunov exponent is the same as the one for $L$, and the conclusion is that the conjugacy is $C^1$. Also a partial result in this direction was obtained by Poletti in \cite{Pol18}.

A similar rigidity result is true in higher dimensions, but we have to make some restrictions. First, since there are no classification results in higher dimensions, we have to restrict our attention to a $C^1$ neighborhood of the Anosov automorphism $L$.

Second, we need the linear Anosov diffeomorphism to be \emph{irreducible}, that is, it has no rational invariant subspaces, or, its characteristic polynomial is irreducible over $\mathbb{Q}$. It is necessary to ask some extra condition in higher dimension, the irreducibility of $L$ being an example, otherwise there are counterexamples constructed by de la Llave and Gogolev \cite{L92,G08}. 

We will also ask that $L$ has simple real eigenvalues with distinct absolute values, in order to avoid complications that may arise from having multiple exponents.

\begin{main}\label{main.higherdimension}

Let $L$ be an irreducible Anosov automorphism of $\TT^d$ with simple real eigenvalues with distinct absolute values. If $f$ is a $C^2$ volume preserving diffeomorphism $C^1$ close to $L$ and has the same Lyapunov exponents, then $f$ is $C^{1+\vep}$ conjugated to $L$ for some $\vep>0$.

\end{main}

The condition of being $C^1$ close is needed in order to ensure the existence of the dominated decomposition for $f$ and the existence of corresponding one dimensional foliations, similar to the automorphism $L$. We also need that the conjugacy with the linear map preserves the weak stable and weak unstable foliations. We remark that the $C^1$ neighborhood that we need is in fact fairly large, by the results in \cite{FPS} it is enough to have an isotopy between $L$ and $f$ with diffeomorphisms preserving the dominated splitting into one-dimensional sub-bundles.

Reminding the parallel with the rigidity in Hyperbolic Geometry, the equivalent of our results would be the so-called {\it Entropy Rigidity Conjecture}. It was initially proposed by Katok in \cite{Ka}, where he also proved it in dimension two: if the measure of maximal entropy and the Liouville measure of the geodesic flow coincide (or equivalently if the metric entropy of the Liouville measure is the topological entropy), then the manifold must be locally symmetric.

The conjecture was later extended by Sullivan, Kaimanovich, in order to include other types of invariant measures (like the harmonic measure). There are some partial results in this direction by Flaminio \cite{Fl}, and some versions of the conjecture for Anosov flows by Foulon \cite{Fo}. Some good surveys on the topic of rigidity in Dynamics and Hyperbolic Geometry can be found in \cite{Fe,Has,Sp}.

In our case, since the Anosov diffeomorphisms have in general different rates of expansion or contraction in different directions, the equivalent of the Entropy Rigidity Conjecture would have to take into account all this directions, which in our case are represented by the one dimensional invariant foliations.

The topological entropy along each of these foliations is exactly the Lyapunov exponent of the linear map, or the logarithm of the absolute value of the corresponding eigenvalue (see \cite{Sa,SX09} for example). The place of the Liouville measure is taken by the volume, and the Lyapunov exponent of the volume along an one dimensional foliation replaces in our case the metric entropy. It is worth mentioning that the Lyapunov exponent and the metric entropy are related by the Ruelle inequality, and if the measure is absolutely continuous along the foliation (which turns out to happen in our case), then the Lyapunov exponent and the metric entropy are equal by the Pesin Formula.

Given an invariant measure for a diffeomorphism $f$, and a foliation $\cF$ with uniformly $C^1$ leaves which is preserved and uniformly expanded by $f$, one can define a {\it conditional metric entropy of $f$ and $\mu$ relative to $\cF$}, $h_{\mu}(f,\cF)$, for more details see Section 2. One also has a notion of measure which is absolutely continuous along the foliation, called Gibbs expanding state, see again Section 2 for more details. An alternative reformulation of Theorem \ref{main.higherdimension} is the following corollary, which probably shows better the connection of our result with the Entropy Rigidity Conjecture:

\begin{maincor}\label{maincor.ECvolume}
Let $L$ be an irreducible Anosov automorphism of $\TT^d$ with simple real eigenvalues with distinct absolute values. Suppose that $f$ is a $C^2$ diffeomorphism $C^1$ close to $L$, and one of the following holds:
\begin{enumerate}
\item
$f$ is volume preserving and the conditional metric entropies of $f$ and the volume relative to the one dimensional invariant foliations of $f$ are equal to the Lyapunov exponents of $L$ (i.e. the volume has maximal conditional entropy along each one dimensional foliation);
\item
the conditional metric entropies of $f$ and the Gibbs expanding states relative to the one dimensional invariant foliations of $f$ are equal to the Lyapunov exponents of $L$ (i.e. the Gibbs expanding states have maximal conditional entropy along each one dimensional foliation).
\end{enumerate}
Then $f$ is $C^{1+\vep}$ conjugated to $L$ for some $\vep>0$.
\end{maincor}

\subsection{Partially hyperbolic diffeomorphisms}

In the case of partially hyperbolic diffeomorphism we do not have in general the structural stability, because of the lack of hyperbolicity in the center direction. However there exists a weak form of structural stability: under some conditions on the center foliation, then every $C^1$ perturbation is {\it leaf conjugate} with the initial system, meaning that the 'conjugating' homeomorphism preserves the original dynamics modulo the center leaves.

In general one cannot hope that the dynamics on the center leaves is preserved by perturbations. This is why the usual notion of rigidity for partially hyperbolic diffeomorphisms is the regularity of the leaf conjugacy with the regular model, and in particular one should have some regularity of the center foliation.

Let us comment more about partially hyperbolic diffeomorphisms. By the stable manifold theorem, the stable and unstable bundles of a partially hyperbolic diffeomorphism are uniquely integrable: they are tangent to the unique stable and unstable foliations respectively. But the center bundle is not always integrable even if $\dim E^c=1$ (see \cite{HHUcoh}).

In this paper, we only consider the partially hyperbolic diffeomorphisms which are \emph{dynamically coherent}, that is, the center bundle $E^c$, center stable bundle $E^{cs}=E^s\oplus E^c$ and the center unstable bundle $E^{cu}=E^c\oplus E^u$ are tangent to invariant foliations: the center foliation $\cF^{c}$, the center stable foliation $\cF^{cs}$ and the center unstable foliation $\cF^{cu}$ respectively. By \cite{HPS70,HPS77}, for a dynamically coherent partially hyperbolic diffeomorphism $f$, if the center foliation of $f$ is \emph{plaque expansive}, then every diffeomorphism $C^1$ close to $f$ is also dynamically coherent, and is leaf conjugate to $f$. Moreover, the center foliation is plaque expansive if the center foliation is $C^1$.

The center foliation, when it exists, in general is not smooth, and indeed, it may have a much more complicated behavior than the stable and unstable foliations. As observed in \cite{A,AS,BP74,PSh72}, the stable and unstable foliations are always absolutely continuous, that is, the disintegration of Lebesgue measure of the manifold along the leaves is equivalent to the Lebesgue measure of the leaves for almost every leaf. But this in general is false for the center foliation of volume preserving partially hyperbolic diffeomorphisms.

The first counterexample was constructed by Katok in \cite{M}, for a $C^2$ family of 2 dimensional volume preserving $C^2$ Anosov diffeomorphisms $\gamma_t$ ($t\in I=[0,1]$). The map $\gamma(x,t)=(\gamma_t(x),t)$ defines a $C^2$ partially hyperbolic diffeomorphism on $\TT^2\times I$, and this partially hyperbolic diffeomorphism is dynamically coherent. Denote by $h_t$ the unique conjugacy, in the isotopy class of identity map, between $f_0$ and $f_t$, then the center leaf passing through $(x,0)$ is $\cF^c(x,0)=\{(h_t(x),t):\ t\in I\}$. If a family of fixed points $p_t$ of $f_t$, which depends smoothly on $t\in I$ (it is in fact a center leaf) has the property that the derivative $D\gamma_t(p_t)$ has different eigenvalues for different values of $t\in I$, then Katok shows that the center foliation $\cF^c$ is not absolutely continuous, and furthermore there exists a full volume subset of $\TT^2\times I$ which intersects every center leaf in a unique point.

Many subsequent results showed that for volume preserving diffeomorphisms the non-absolute continuity of the center foliation is indeed a quite general phenomena, for instance see \cite{RW,SW,SX09,HP,AVW1,AVW2,VY2}. There is a different situation for diffeomorphisms which do not preserve the volume, see Viana and Yang \cite{VY1}.

This is why it is interesting to find conditions that give the regularity of the center foliation, and there are only few results in this direction. For example, for the volume preserving skew product case, it was shown in \cite{AVW2} that if all the center Lyapunov exponents are vanishing almost every where, then the center foliation is absolutely continuous, and with an additional assumption (for example, if the center leaf is homeomorphic to a circle or to a two dimensional sphere), then the center foliation is indeed smooth.

In \cite{AVW1} similar results are obtained for perturbations of time one maps of geodesic flows of hyperbolic surfaces: if the center exponent vanishes, and the center foliation is absolutely continuous, then not only the center foliation is smooth, but the perturbation must be also the time one map of an Anosov flow. This result was generalized by Butler and Xu \cite{BX} for perturbations of geodesic flows on higher dimensional manifolds with constant negative curvature (see also \cite{Bu}).

In these rigidity-type results, one assumes that the center Lyapunov exponents are zero, eventually together with other conditions like accessibility and center bunching, and uses an {\it Invariance Principle} in order to obtain the regularity of the center foliation. Unlike the results described above, we can obtain rigidity results for partially hyperbolic diffeomorphisms under some conditions on the stable and unstable Lyapunov exponents, and without any condition on the center, nor on accessibility or center bunching.

The basic ideas behind the two types of rigidity results which we obtain, the ones for Anosov diffeomorphisms and the ones for partially hyperbolic diffeomorphisms, are very similar, the only difference is that instead of obtaining smooth conjugacy, we will obtain smooth center holonomy. In fact one can view the conjugacy between Anosov diffeomorphisms as a particular case of a center holonomy for a partially hyperbolic diffeomorphism.

We first illustrate our ideas with the following example which is similar to the Katok's construction, and which is a corollary of Theorem \ref{main.higherdimension}.

\begin{maincor}\label{maincor.Katokexample}

Let $L$ be an irreducible linear hyperbolic automorphism of $\TT^d$ with simple real eigenvalues with distinct absolute values. Consider a $C^2$ family of $C^2$ volume preserving diffeomorphisms $\gamma_t$ ($t\in I=[0,1]$) inside a $C^1$ small neighborhood of $L$, such that the diffeomorphisms $\gamma_t$ satisfy the hypothesis of Theorem \ref{main.higherdimension}. Then the $C^2$ volume preserving diffeomorphism $\gamma: \TT^d\times I\to \TT^d\times I$, $\gamma(t,x)=(\gamma_t(x),t)$ is partially hyperbolic and dynamically coherent.

If for every $t\in I$, the $d$ exponents of $\gamma_t$ coincide with the exponents of $L$, or equivalently if the average exponents of $\gamma$ in the stable and unstable direction coincide with the exponents of $L$, then the center foliation of $\gamma$ is $C^{1+\vep}$.

In the opposite case, if for different values of $t\in I$, the sum of the unstable Lyapunov exponent of $\gamma_t$ takes different values, then there is a full volume subset of $\TT^d\times I$ which intersects every center leaf in a unique point.

\end{maincor}

By average exponents we mean the integral of the exponents with respect to the volume (in this case the volume is not ergodic, so the exponents may depend on the point). This result can be extended to the more general case of perturbations of skew products over Anosov diffeomorphisms.

Remember that, if $f_0\in\Diff^1(\TT^d\times N): f_0(x,y)=(L(x),h_x(y))$ is a partially hyperbolic diffeomorphism with $\{\cdot\}\times N$ to be the center foliation, then $f$ is plaque expansive, and for any $f$ $C^1$ close to $f_0$, $f$ is dynamically coherent.

\begin{main}\label{main.cneterfible}

Let $L$ be a two dimensional linear Anosov map on $\TT^2$, let $N$ be a compact Riemannian manifold without boundary,
and let $f_0: \TT^2\times N \to \TT^2\times N$ be a volume preserving partially hyperbolic skew product diffeomorphism
$$
f_0((x,y))=(L(x),h_x(y)),
$$
with $\{\cdot\}\times N$ corresponding to the center foliation. If $f\in \Diff^2(\TT^2\times N)$ is sufficiently $C^1$ close to $f_0$, preserves the volume, and the average strong stable and strong unstable Lyapunov exponents of $f$ coincide with the exponents of $L$, then the center foliation of $f$ is $C^{1+\vep}$ for some $\vep>0$, and $f$ is $C^{1+\vep}$ conjugated to a true skew product over $L$. Moreover, if $N=S^1$, and $f$ is $C^{\infty}$, then $f$ is $C^\infty$ conjugated to a true skew product over $L$ with rotations on the center fibers: a diffeomorphism $f^\prime: \TT^2\times S^1 \to \TT^2\times S^1$,
$$
f^\prime((x,y))=(L(x),y+\alpha(x))
$$
for some $C^{\infty}$ function $\alpha:\TT^2\rightarrow S^1$.

\end{main}

The regularity of the center foliation can be obtained under weaker conditions, $f$ has to be of skew product type over the Anosov automorphism of the two-torus $L$. We say that a partially hyperbolic diffeomorphism $f$ on $M$ is \emph{skew product type over the Anosov diffeomorphism (or more generally homeomorphism) $\overline f$} on $\overline M$ if $f$ is dynamically coherent, the center foliation forms a continuous fiber bundle over $\overline M$ with compact fiber, and the continuous projection $\pi$ from $M$ to $\overline M$ along the center foliation semiconjugates $f$ and $\overline f$.

In general, if $f:M\rightarrow M$ is a partially hyperbolic diffeomorphism which is dynamically coherent and the center foliation $\cF^c$ forms a fiber bundle with compact fiber, then the map $\overline f$ induced by $f$ on the quotient space $\overline M=M/\cF^c$ is an Anosov homeomorphism (see Viana~\cite{Almost}, Section 2.2 of \cite{VY1} and Section 4 of \cite{G}), in particular, $\overline{f}$ is an expansive homeomorphism.

If furthermore $\dim(E^u)=\dim(E^s)=1$, then the quotient space is a two dimensional surface. By a generalization of theorems of Franks~\cite{Fra70} and Newhouse~\cite{N70}, or making use of the topological classification of expansive homeomorphisms of surfaces by Lewowicz \cite{Lew} and Hiraide \cite{H}, one can show that the quotient space $\overline{M}$ is homeomorphic to a two torus, and moreover, $\overline{f}$ is conjugated to an Anosov automorphism of the torus $L$.

We can generalize the Theorem \ref{main.cneterfible} to partially hyperbolic diffeomorphisms with higher dimensional Anosov base, but unfortunately we have to assume a stronger condition on the Anosov part. Remember that given an invariant measure for a diffeomorphism $f$, and a foliation $\cF$ with uniformly $C^1$ leaves which is preserved and uniformly expanded by $f$, one can define a conditional metric entropy of $f$ and $\mu$ relative to $\cF$, for more details see Section 2. We have the following result.

\begin{main}\label{main.higherbase}

Let $L$ be a linear Anosov diffeomorphism over $\TT^d$, irreducible and with simple real eigenvalues with distinct absolute values, $N$ is a compact Riemannian manifold without boundary, and $f_0: \TT^d\times N \to \TT^d\times N$ be a volume preserving partially hyperbolic skew product diffeomorphism
$$
f_0((x,y))=(L(x),h_x(y)),
$$
with $\{\cdot\}\times N$ corresponding to the center foliation. Suppose that $f\in \Diff^2(\TT^d\times N)$ preserves the volume and is sufficiently $C^1$ close to $f_0$, such that the unstable and stable foliations of $f$ are subfoliated by one dimensional expanding foliations for $f$ respectively for $f^{-1}$. If the conditional entropies along these one dimensional foliations coincide with the exponents of $L$, then the center foliation of $f$ is $C^{1+\vep}$ for some $\vep>0$, and $f$ is $C^{1+\vep}$ conjugated to a true skew product over $L$.

\end{main}

\subsection{Ideas of the proofs}

Let us comment on the ideas of the proofs.

The first remark is that the hypothesis of our theorems imply that the stable and unstable foliations of $f$ are subfoliated by one dimensional foliations. By Journ\'{e}'s regularity result \cite{J}, in order to prove that the conjugacy between two Anosov maps is $C^{r}$, it is sufficient to show that the conjugacy restricted to every one dimensional subfoliation is $C^r$.

For the partially hyperbolic case, in order to show that the center foliation is $C^r$, it is sufficient to show that every center leaf is $C^r$, and the center holonomy is (uniformly) $C^r$. Again by a Journ\'e-type regularity result, if we can show the regularity of the holonomy along all the one dimensional stable and unstable subfoliations, this will imply that the center holonomy is $C^r$. This is a classical technique used in various other papers on rigidity.

The second remark is that there is a classical method from measure theory which can be used to show the smoothness of maps between one dimensional (sub)manifolds. If $h$ is a one dimensional homeomorphism between the $C^r$ curves $I_1$ and $I_2$, and there exist probability measures $\mu_i$ on $I_i$, absolutely continuous with respect to the Lebesgue measures on $l_i$, with $C^{r-1}$ densities $\rho_i$ bounded away from zero and infinity, and such that $h_*(\mu_1)=\mu_2$, then $h$ is a $C^r$ diffeomorphism between $l_1$ and $l_2$. This observation has been widely used in a series of works, see for instance \cite{L92,AVW1}.

A third remark is that the one dimensional stable and unstable subfoliations which we have are either uniformly expanding or uniformly contracting (or uniformly expanding for $f^{-1}$). This property allows us to construct invariant measures which are absolutely continuous (with respect to Lebesgue) along the one dimensional subfoliations, which we call Gibbs expanding states, because they generalize the classical notion of Gibbs u-state. The properties of the Gibbs expanding states are similar to the properties of the Gibbs u-states (see for instance \cite{Beyond}[Chapter 11]), in particular we have that the densities (with respect the Lebesgue) of the disintegrations of a Gibbs expanding state along the corresponding expanding foliation are uniformly $C^{r-1}$, and can be extended continuously to the support of the measure.

Finally, our results would follow if we can establish that the conjugacy for the Anosov case, or the central holonomy for the partially hyperbolic case, preserves the one dimensional subfoliations, and preserves the corresponding Gibbs expanding states, which have full support. There are different arguments that we use for different one dimensional subfoliations.

In the Anosov case, we use the main technical result Theorem~\ref{main.technique}, which is presented in the next section, and which may have its own interest. It is a detailed characterization of the situation when two Gibbs expanding states are preserved by a conjugacy, including a characterization in terms of Lyapunov exponents ([B4]). It is build on previous work of Ledrappier, Young, and others, and a crucial ingredient is the Pesin formula relative to an expanding foliation $\cF$: the conditional entropy of the invariant measure $\mu$ along $\cF$ is equal to the Lyapunov exponent of $\mu$ along $\cF$ if and only if $\mu$ is a Gibbs expanding state for $\cF$. We will apply Theorem~\ref{main.technique} in order to obtain Theorem~\ref{main.3d}, Theorem~\ref{main.higherdimension}, Corollary \ref{maincor.ECvolume}, Corollary~\ref{maincor.Katokexample} and Corollary~\ref{maincor.generalconjugation}. It is worth mentioning that our characterization does not depend on the Livshitz theorem and hence could also be applied to partially hyperbolic diffeomorphisms, where Livshitz theorem may not valid, and periodic points may not even exist.

For the case of partially hyperbolic diffeomorphisms we will use a version of the so-called invariance principle, Theorem \ref{p.partialentropy} from Section \ref{s.invarianceprinciple}. In contrast with the standard invariance principle of Avila and Viana \cite{AV}, where in general the vanishing center exponents is needed, this is a different form of invariance principle which uses the concept of conditional entropy along an expanding foliation, and was developed in \cite{TY,VY3}. Theorem \ref{p.partialentropy}, together with the same Pesin formula relative to an expanding foliation, will be used in order to obtain Theorem~\ref{main.cneterfible} and Theorem~\ref{main.higherbase}.

There are some other important steps involved in the proofs, preparing the application of the 2 main technical tools mentioned above. The hypothesis on the Lyapunov exponents and the uniqueness of the measure of maximal entropy for Anosov maps will imply that the conjugacy between an Anosov map and the linear part, or the semiconjugacy between a skew product and the hyperbolic linear part, takes volume to volume. Thus the Gibbs expanding state which we will use in our considerations will be the volume, for all the one dimensional subfoliations. In order to show that the conjugacy or the center holonomy preserves various one dimensional subfoliations, and in order to bootstrap for better regularity, we will use some previous results of Gogolev and others (\cite{G,GG,GKS}).

\subsection{Organization of the paper}

In the next section, we will give various definitions and we will state our main technical result on the Gibbs expanding states, Theorem ~\ref{main.technique}, together with a corollary characterizing smooth conjugacy between some nonlinear Anosov diffeomorphisms (Corollary~\ref{maincor.generalconjugation}, a generalization of Theorem~\ref{main.higherdimension}). In Section~\ref{s.3danosov} and Section~\ref{s.highdAnosov} we will assume Theorem~\ref{main.technique} and we will give the proofs of Theorem~\ref{main.3d}, Theorem~\ref{main.higherdimension}, Corollary \ref{maincor.ECvolume}, Corollary~\ref{maincor.Katokexample} and Corollary~\ref{maincor.generalconjugation}. Theorem~\ref{main.technique} and Corollary~\ref{maincor.periodicpoints} will be proved in Section~\ref{s.technique}. In Section~\ref{s.invarianceprinciple} we introduce the background on the Invariance principle. Finally, the proof of Theorem~\ref{main.cneterfible} is given in Section \ref{s.skew} and the proof of Theorem~\ref{main.higherbase} appear in Section~\ref{higherbase}. We will end the paper with some further remarks on possible extensions of our results.

\textbf{Acknowledgements:} We would like to thank Mauricio Poletti for conversations on this topic, from where the idea of this work originated, and to Andrey Gogolev for useful comments on an early version of the paper. We would also like to thank the anonymous referees for useful comments and suggestions.

\section{Preliminaries}

In this section we will first introduce some definitions and background results, and then we will present our main technical result.

Throughout this section, we assume $M$ to be a $C^\infty$ Riemannian manifold without boundary, and $f\in \Diff^r(M)$ for some $r\geq 1$.

\subsection{Expanding foliations}
\begin{definition}\label{df.expandingfoliation}

We say that a continuous foliation $\cF$ is an \emph{expanding foliation} of $f$ with $C^r$ leaves if

\begin{enumerate}
\item
the foliation is invariant under the iteration of $f$;
\item
the leaves of $\cF$ are uniformly $C^{r}$, and vary continuously in the $C^1$ topology;
\item
and $Df$ restricted to the tangent bundle of each leaf of $\cF$ is uniformly expanding, meaning that for some Riemannian metric, for some $\lambda>0$, for all $x\in M$, for all unit vectors $v\in T_x\cF$, we have
$$
\|Df|_{T_x\cF}(v)\|\geq e^{\lambda}.
$$

\end{enumerate}

\end{definition}

If no confusing is caused, we will say that $\cF$ is just an expanding foliation.

\subsection{Gibbs expanding states}
Let $\cF$ be a foliation, $\mu$ a Borel measure on $M$, and $B$ a foliation chart of $\cF$ with $\mu(B)>0$. Let $D$ be a transverse disk of the foliation chart, then we have a local projection map $\pi: B \to D$ along the $\cF$ leaves, and a corresponding quotient measure $\overline{\mu}=\pi_*(\mu)$ supported on $D$. Then $D$ can be identified as the space of $\cF\mid_B$ leaves and the measure $\overline{\mu}$ is a measure of the set of leaves. By \cite{Rok49}, the \emph{disintegration of $\mu$ along the foliation $\cF$ on $B$} consists of the probability measures $\{\mu^{\cF}_{x}\}$, supported on $\mu$ almost every leaf $\cF(x)$, such that:

\begin{equation}\label{eq.defdisintegration}
d\mu(y)=\int d\mu^{\cF}_x(y)d\mu(x)=\int_{D} d\mu^{\cF}_{x}(y) d\overline{\mu}(x).
\end{equation}

The measures $\mu^{\cF}_{x}$ exist and are unique (up to $\mu$-zero measure), and are called the disintegrations of $\mu$ along the plaques $\cF(x)$.

\begin{definition}\label{df.Gibbsestate}

Suppose $\cF$ is an expanding foliation for the diffeomorphism $f$. An invariant probability $\mu$ of $f$ is a \emph{Gibbs expanding state along the foliation $\cF$} if for any foliation chart of $\cF$, the disintegration of $\mu$ along the plaques of this chart is equivalent to the Lebesgue measure on the plaques for $\mu$ almost every plaque. If there is no confusion on the foliation $\cF$, we will also call the Gibbs expanding state along $\cF$ just \emph{Gibbs $e$-state}, and we will denote the space of Gibbs expanding states along $\cF$ by $\Gibb^e(f,\cF)$.

\end{definition}

The Gibbs expanding state is a generalization of the Gibbs $u$-state, any Gibbs $u$-state is indeed a Gibbs expanding state when $\cF$ is the corresponding unstable foliation. The following properties are well known for Gibbs $u$-states, see for instance \cite{Beyond}[Section 11]. But the proof can be directly generalized to Gibbs expanding states, because it basically depends on the expanding property of the foliation and on the distortion estimates of the tangent map restricted to $\cF$, so we will not provide a proof here.

\begin{proposition}\label{p.Gibbsestate}

Let $\cF$ be an expanding foliation with uniform $C^r$ ($r>1$) leaves, for the $C^r$ diffeomorphism $f$. Then

\begin{itemize}
\item
[A1] $\Gibb^e(f,\cF)$ is a non-empty convex and compact set with respect to the weak-* topology;
\item
[A2] almost every ergodic component of any Gibbs $e$-state is still a Gibbs $e$-state;
\item
[A3] the support of any Gibbs $e$-state is $\cF$ saturated, that is, its support consists of an union of $\cF$ leaves;
\item
[A4] let $\mu\in \Gibb^e(f,\cF)$, and $B$ any foliation chart of $\cF$, then for $\mu\mid_B$ almost every
$x$, the disintegration of $\mu$ along $\cF\mid_B(x)$ equals to
$\rho(z)d\Leb\mid_{\cF\mid_B(x)}(z)$, where $\rho$ is continuous on $B$ and uniformly $C^{r-1}$ along the plaques of $\cF$, and is given by the formulas
\begin{equation}\label{eq.density}
\rho(z)=\frac{\Delta(x,z)}{\int \Delta(x,z)d\Leb\mid_{\cF\mid_B(x)}(z)}  \text{ and }
\end{equation}
$$
\Delta(x,z)=\lim_{n\to \infty}\frac{\det(D_zf^{-n}\mid_{T_z\cF(x)})}{\det(D_xf^{-n}\mid_{T_x\cF(z)})}.
$$
\end{itemize}

\end{proposition}

The property (A1) above was used in Step 2 of \cite{GG}[Section 4.3] and in \cite{G08} in order to study the center foliation which is uniformly expanding.

\subsection{Conditional entropy along an expanding foliation}
One can define the conditional entropy along an expanding foliations in the following way. The next definitions work also for $C^1$ diffeomorphisms and foliations.

\begin{definition}

Let $\cF$ be an expanding foliation of a diffeomorphism $f$, and $\mu$ an invariant probability measure. A measurable partition $\cA$ is called \emph{subordinate} to $(\cF,f,\mu)$ if:

\begin{enumerate}
\item
for $\mu$ almost every point $x\in M$, $\cA(x)\subset \cF(x)$ and $\cA(x)$ contains a neighborhood of $x$,
\item
$f(\cA)\preceq \cA$,
\item
and $\vee_{i=0}^{\infty} f^{-i}(\cA)$ is the partition into points.
\end{enumerate}

\end{definition}

The existence of such partitions subordinate to Pesin's unstable lamination was proved by Ledrappier and Strelcyn in \cite{LS}. For a general expanding foliation the existence of such a subordinated partition, with pieces uniformly bounded from above, was shown by the second author in (\cite{Yang16}).

Assume that $\xi$ is a subordinated partition for $(\cF,f,\mu)$, then we may define the \emph{partial or conditional entropy along the expanding foliation $\cF$}:
$$
h_\mu(f,\cF)=H_\mu(f^{-1}(\xi) \mid \xi)=\int_M-\log\mu_x^{\cF}(f^{-1}\xi(x))d\mu(x).
$$
The partial entropy is finite, and the definition does not depend on the choice of the subordinated partition, for a proof see \cite{LY1}[Lemma 3.1.2]. It is affine and upper semicontinuous (as a function of the measure). For more on conditional entropy see \cite{LS}, \cite{LY1}, \cite{Yang16}, also \cite{HHW} for alternative definitions and various properties. 

The conditional entropy along an expanding foliation satisfies various good properties, which we know from the classical ergodic theory, let us list some of them:
\begin{itemize}
\item {\bf The variation principle relative to a foliation} (see \cite{HHW}).\\
One can define the {\it topological entropy $h(f,\cF)$ along an expanding foliation} $\cF$ as the supremum of the Bowen entropy for disks inside leaves of the foliation. This turns out to be equal to the so-called {\it volume growth} of the foliation introduced by Yomdin and Newhouse, i.e. the maximal exponential rate of growth of the volume of disks inside the leaves of the foliation. Then for a $C^1$ diffeomorphism $f$ we have
\begin{eqnarray*}
h(f,\cF)&=&\sup\{ h_{\mu}(f,\cF):\ \mu \hbox{ invariant probability for }f\}\\
&=&\sup\{ h_{\mu}(f,\cF):\ \mu \hbox{ ergodic invariant probability for }f\}
\end{eqnarray*}
\item {\bf The Shannon-McMillan-Breiman Theorem} (see \cite{HHW}, \cite{LY2}).\\
Let $B_n(x,f,\cF,\delta)=\cap_{i=0}^nf^{-i}(B_{\cF}(x,\delta))$ be the dynamical balls of $f$ restricted to the foliation $\cF$. Assume that $f$ is $C^1$ and $\mu$ is ergodic, then for $\mu$-almost every $x\in M$ we have
$$
h_{\mu}(f,\cF)=\lim_{n\to\infty}\mu_x^{\cF}(f^{-n}\xi(x))=\lim_{n\to\infty}\mu_x^{\cF}(B_n(x,f,\cF,\delta)).
$$
\item {\bf The Ruelle inequality.}\\
Several Lyapunov exponents of $f$ with respect to $\mu$ will correspond to the tangent bundle of $\cF$; let us denote by $\lambda^{\cF}(x)$ their sum, and we call it the {\it Lyapunov exponent of $\mu$ along $\cF$}. By Birkhoff Ergodic Theorem we have
\begin{equation}\label{eq.exponentfoliation}
\lambda^{\cF}(f,\mu)=\int_M\lambda^{\cF}(x)d\mu(x)=\int_M\log(\det(Df\mid_{T_x\cF}))d\mu(x).
\end{equation}
If $f$ is $C^1$ then we have
\begin{equation}\label{eq.Ruelle}
h_{\mu}(f,\cF)\leq\lambda^{\cF}(f,\mu).
\end{equation}
\item {\bf The Pesin formula and its converse.} (see \cite{L},\cite{LY1})\\
Assume that $f$ is $C^r$ and $\cF$ has uniformly $C^r$ leaves for some $r>1$. Then $\mu$ is a Gibbs expanding state for $\cF$ if and only if the following equality ({\it Pesin Formula}) holds:
$$
h_{\mu}(f,\cF)=\lambda^{\cF}(f,\mu).
$$ 
\end{itemize}

\subsection{Ruelle inequality and Pesin formula}
Our proofs rely heavily on the last two results, the Ruelle inequality and the Pesin formula. The proof of the Pesin formula was given by Ledrappier~\cite{L} (and an explanation of the proof can be also found in \cite{LY1}[Lemma 6.13]) in a slightly different context, for the unstable lamination, however the proof can be translated directly to our context. Since we are not aware of a reference in the context of expanding foliations which are not the unstable foliation, we will give a sketch of this proof following closely the work of Ledrappier, together with a proof of the Ruelle inequality for the $C^1$ case based on a slight modification of the same argument.

\begin{theorem}[Ruelle inequality and Pesin formula]\label{l.criterionofgibbs}

Let $\cF$ be an expanding foliation with $C^r$ leaves of a $C^r$ diffeomorphism $f$, $r\geq1$ and $\mu$ an invariant measure for $f$. Let $\lambda^{\cF}(f,\mu)$ be the Lyapunov exponent of $\mu$ along the tangent bundle of $\cF$. Then
$$
h_\mu(f,\cF)\leq\lambda^{\cF}(f,\mu).
$$
Furthermore, if $r>1$, then $\mu$ is a Gibbs $e$-state of $f$ if and only if
$$
h_\mu(f,\cF)=\lambda^{\cF}(f,\mu).
$$

\end{theorem}

\begin{proof}

Let $\xi$ be a measurable partition subordinate to $(\cF,f,\mu)$. In order to simplify the proof we assume that the elements of $\xi$ are uniformly bounded from above (for such a construction see \cite{Yang16}). Then by definition,
$$h_\mu(f,\cF)=H_\mu(f^{-1}(\xi)\mid \xi).$$

Let $J^{\cF}(x)=\det(Df\mid_{T_x\cF})$. For $\mu$ almost every $x$, denote by $\Leb_x^{\cF}$ the restriction of Lebesgue measure on the element $\xi(x)$, and we define a density function on $\xi(x)$: $\rho_n^x(y)=\frac{\Delta_n(x,y)}{L_n(x)}$ where 
$$\Delta_n(x,y)=\frac{\prod_{i=1}^n{J^{\cF}(f^{-i}(x))}}{\prod_{i=1}^n J^{\cF}(f^{-i}(y))}\in(0,\infty)$$
and $L_n(x)=\int_{\xi(x)} \Delta_n(x,y)d\Leb^{\cF}_x(y)\in(0,\infty)$.

If $f$ and the leaves of $\cF$ are uniformly $C^{1+\alpha}$, one can obtain $\Delta,L$ and $\rho$ as limits when $n$ goes to infinity of $\Delta_n,L_n$ and $\rho_n$:
$$
\Delta(x,y)=\lim_{n\rightarrow\infty}\Delta_n(x,y),\ L(x)=\lim_{n\rightarrow\infty}L_n(x),\ \rho(y)=\lim_{n\rightarrow\infty}\rho^x_n(y).
$$

Since the diameter of $\xi(x)$ is uniformly bounded from above, we get the following lemma:
\begin{lemma}\label{boundeddensity}
For any natural number $n$, the function $(x,y)\to \Delta_n(x,y)$ is uniformly continuous on $M\times M$, and $(x,y)\to \Delta_n(x,y)$ is bounded away from $0$ and $\infty$ for $y\in\xi(x)$. The function $x\to L_n(x)$ is measurable and bounded away from $\infty$.

If $f$ and the leaves of $\cF$ are uniformly $C^{1+\alpha}$, then the bounds are also uniform with respect to $n$, so they hold also for $\Delta$ and $L$, meaning that $\Delta$ is bounded away from $0$ and $\infty$, while $L$ is strictly positive and bounded away from $\infty$.
\end{lemma}

We define the family $\{\nu_n^x\}$ of probability measures on the elements of the partition $\xi$, absolutely continuous with respect to Lebesgue on the leaves of $\cF$ and with density $\rho_n$:
$$
d\nu_n^x=\rho_n^x d\Leb^{\cF}_x.
$$
We have

\begin{lemma}\label{l.refine}
$$\lim_{n\to\infty}\int - \log \nu_n^x(f^{-1}\xi(x))d\mu(x)=\int \log J^{\cF} d\mu.$$
\end{lemma}

\begin{proof}
Define $q_n(x)=\nu_n^x((f^{-1}\xi)(x))$, then
\begin{eqnarray*}
q_n(x)&=&\frac 1{L_n(x)}\int_{(f^{-1}\xi)(x)}\Delta_n(x,y)d\Leb^{\cF}_x(y)\\
&=&\frac 1{L_n(x)}\int_{(f^{-1}\xi)(x)}\frac{\prod_{i=1}^n{J^{\cF}(f^{-i}(x))}}{\prod_{i=1}^n J^{\cF}(f^{-i}(y))}d\Leb^{\cF}_x(y)\\
&=&\frac 1{L_n(x)J^{\cF}(x)}\int_{(f^{-1}\xi)(x)}\Delta_n(f(x),f(y))J^{\cF}(y)\frac{J^{\cF}(f^{-n}(x))}{J^{\cF}(f^{-n}(y))}d\Leb^{\cF}_x(y)\\
&=&\frac 1{L_n(x)J^{\cF}(x)}\int_{\xi(f(x))}\Delta_n(f(x),y')\frac{J^{\cF}(f^{-n}(x))}{J^{\cF}(f^{-n-1}(y'))}d\Leb^{\cF}_{f(x)}(y')\\
&=&\frac{L_n(f(x))}{L_n(x)}\frac{1}{J^\cF(x)}A_n(x),
\end{eqnarray*}
for some measurable function $A_n:M\to\mathbb R_+$ which converges uniformly to 1 when $n$ goes to infinity. This is because the quotient $\frac{J^{\cF}(f^{-n}(x))}{J^{\cF}(f^{-n-1}(y'))}$ converges uniformly to 1, since $y'$ and $f(x)$ are in the same unstable piece, and we assumed that the unstable pieces are bounded.

Recall that $L_n$ is a positive finite-valued measurable function. The above equation shows also that $\frac{L_n(f(x))}{L_n(x)}$ is bounded from above, since $q_n\leq 1$:
$$\frac{L_n(f(x))}{L_n(x)} =  q_n(x)\cdot\frac {J^{\cF}(x)}{A_n(x)} <C<\infty.$$
The bound is uniform on $n$, and if $r>1$, it also holds for $L$.

In particular we have that $\log^+\left(\frac{L_n(f(x))}{L_n(x)}\right)$ is integrable, and well-known ergodic results (for example Proposition 2.2 in \cite{LS}) show that $\log\left(\frac{L_n(f(x))}{L_n(x)}\right)$ must be also integrable and
$$\int_M\log\left(\frac{L_n(f(x))}{L_n(x)}\right)d\mu=0.$$

It follows that $\log q_n$ is integrable and
$$
\int_M \log q_nd\mu=-\int_M \log J^\cF d\mu+\int_M\log A_nd\mu.
$$
Taking the limit when $n$ goes to infinity gives the desired equality.
\end{proof}

For $\mu$ almost every $x$, $(f^{-1}\xi)\mid \xi(x)$ is a countable partition. On each piece $\xi(x)$ we define a step function $\frac{d\nu_n}{d\mu}\mid_{f^{-1}\xi}$ in the following way. For $y\in \xi(x)$, define
$$\frac{d\nu_n}{d\mu}\mid_{f^{-1}\xi}(y)=\frac{\nu_n^x((f^{-1}\xi)(y))}{\mu_x((f^{-1}\xi)(y))}.$$
Since $\nu_n^x$ is a probability, it is easy to see from the construction that 
$$
\int_{\xi(x)}\frac{d\nu_n}{d\mu}\mid_{f^{-1}\xi}d\mu_x=1.
$$
In this way we obtain the real function $\frac{d\nu_n}{d\mu}\mid_{f^{-1}\xi}:M\to\mathbb R_+$ which is well defined almost everywhere and measurable, constant on the pieces of $f^{-1}\xi$, and
$$
\int_M\frac{d\nu_n}{d\mu}\mid_{f^{-1}\xi}d\mu=1.
$$
Applying Jensen inequality in the convex setting we obtain
$$\int_M \log \left(\frac{d\nu_n}{d\mu}\mid_{f^{-1}\xi}\right) d\mu \leq \log \int_M \left(\frac{d\nu_n}{d\mu}\mid_{f^{-1}\xi}\right)d\mu=0,$$
with $\int_M \log \left(\frac{d\nu_n}{d\mu}\mid_{f^{-1}\xi}\right) d\mu=0$ if and only if $\frac{d\nu_n}{d\mu}\mid_{f^{-1}\xi}=1$
almost everywhere.

This means that
\begin{eqnarray*}
h_{\mu}(f,\cF)&=&-\int_M \log \mu_x((f^{-1}\xi)(x))d\mu\\
&\leq& -\int_M \log \nu_n^x((f^{-1}\xi)(x)) d\mu(x)
\end{eqnarray*}

Taking the limit when $n$ goes to infinity and applying Lemma \ref{l.refine} we obtain the desired inequality.

If furthermore $f$ and the leaves of $\cF$ are uniformly $C^{1+\alpha}$, we have that the equality above ($\nu_n$ replaced by $\nu)$ holds if and only if $\mu=\nu$ on $\mathcal{B}_{f^{-1}(\xi)}$, the $\sigma$-algebra generated by $f^{-1}(\xi)$. Similarly, we can show that $\mu=\nu$ on $\mathcal{B}_{f^{-n}(\xi)}$ for any $n>0$, and since $\vee_{i=0}^{\infty} f^{-i}(\cA)$ is the partition into points, we get that $\mu=\nu$ and thus is a Gibbs expanding state.
\end{proof}

\subsection{Gibbs e-states and conjugacy}

Our next result, which will be proved in Section \ref{s.technique}, is the following characterization of the fact that two Gibbs $e$-states are preserved by a conjugacy. 

\begin{main}\label{main.technique}

Let $f,g\in \Diff^k(M)$ ($k\geq 2$), $\cF^i$ be expanding foliations of $i\in\{f,g\}$ with uniform $C^r$ ($r>1$) leaves, and $\mu\in \Gibb^e(f,\cF^f)$. Suppose $f$ and $g$ are conjugated by a homeomorphism $h$, and $h$ maps the foliation $\cF^f$ to the foliation $\cF^g$. Then the following conditions are equivalent:

\begin{itemize}
\item
[B1] $\nu:=h_*(\mu)$ is a Gibbs $e$-state of $g$ for the foliation $\cF^g$;
\item
[B2] $h|_{\cF^f}$ is absolutely continuous on the support of $\mu$ (with respect to Lebesgue on $\cF^f$ and $\cF^g$), with the Jacobian continuous on $\supp(\mu)$, and bounded away from zero and infinity;
\item
[B3] there exists $K>0$, such that for any $x\in \supp(\mu)$ and any integer $n>0$,
$$
\frac{1}{K}<\frac{\det(Df^n\mid_{T_x\cF^f(x)})}{\det(Dg^n\mid_{T_{h(x)}\cF^g(h(x))})}<K
$$
\item
[B4] the sum of the Lyapunov exponents along the expanding foliations are the same: if $\lambda^{\cF^f}(f,\mu)$ is Lyapunov exponent of $\mu$ for $f$ along the foliation $\cF^f$, and $\lambda^{\cF^g}(g,\nu)$ is the Lyapunov exponent of $\nu$ for $g$ along the foliation $\cF^g$ (see \eqref{eq.exponentfoliation}), then
$$
\int \log(\det(Df\mid_{T\cF^f}))d\mu=\lambda^{\cF^f}(f,\mu)=\lambda^{\cF^g}(g,\nu)=\int \log(\det(Dg\mid_{T\cF^g}))d\nu.
$$

\end{itemize}

If the foliations $\cF^f$, $\cF^g$ are 1 dimensional, then the above conditions are also equivalent to

\begin{itemize}
\item[B5$^\prime$] $h$ restricted on each $\cF^f$ leaf within the support of $\mu$ is uniformly $C^{r}$ smooth.
\end{itemize}

\end{main}

We also have the following Livshitz-type result, involving the periodic data.

\begin{maincor}\label{maincor.periodicpoints}

With the same hypothesis as Theorem~\ref{main.technique}, suppose $p\in \supp(\mu)\subset M$ is a periodic point of $f$ with
period $\pi(p)$, then any of the conditions B1-B4 implies that
$$
\det(Df^{\pi(p)}\mid_{T_p\cF^f(p)})=\det(Dg^{\pi(p)}\mid_{T_{h(p)}\cF^g(h(p))}).
$$

Conversely, suppose that there exists a sequence of periodic points $p_n$ of $f$ with period $\pi(p_n)$, and for each $n\in\mathbb N$ let $\mu_n$ be the invariant measure of $f$ supported on $Orb_f(p_n)$. If $\mu$ belongs to the convex closure of $\{\mu_n:\ n\in\mathbb N\}$, and
$$
\det(Df^{\pi(p_n)}\mid_{T_{p_n}\cF^f(p_n)})=\det(Dg^{\pi(p_n)}\mid_{T_{h(p_n)}\cF^g(h(p_n))}) \text{ for each }n\in\mathbb N
$$
then $\nu\in \Gibb^e(g,\cF^g)$.

\end{maincor}

\begin{remark}\label{rk.historyoftechnique}

When $f$ and $g$ are Anosov diffeomorphism, $\cF^f$ and $\cF^g$ are the corresponding unstable foliations, $\mu$ is the volume measure of $f$, and $h$ denotes the conjugation between $f$ and $g$, then property (B1) above means that $h$ is absolutely continuous. When the manifold is two dimensional, the fact that (B1) is equivalent to (B3) was observed in \cite{L92}, and Corollary~\ref{maincor.periodicpoints} was proved in \cite{LMM}.

\end{remark}

Theorem~\ref{main.technique} enables us to generalize Theorem~\ref{main.higherdimension} in order to give a general criterion for the smoothness of the conjugacy between two general Anosov diffeomorphisms (possibly none of them is linear). Let $L$ be an irreducible linear Anosov diffeomorphism over the torus $\TT^d$ with $d$ real eigenvalues of different norm, denote by $\lambda^s_{l}<\lambda^s_{l-1}<\cdots<\lambda^s_1<0<\lambda_1^u<\cdots<\lambda^u_k$ the exponents of $L$ and by $E^s_{l}\oplus E^s_{l-1}\oplus \cdots\oplus E^s_1\oplus E^u_1\oplus \cdots\oplus E^u_k$ the invariant splitting corresponding to the eigenvectors. Suppose that $f$ and $g$ are $C^2$ diffeomorphisms, which are $C^1$ close to $L$, and let $h$ be the conjugacy between $f$ and $g$. Let $E^{s,\sigma}_{l}\oplus E^{s,\sigma}_{l-1}\oplus \cdots\oplus E^{s,\sigma}_1\oplus E^{u,\sigma}_1\oplus \cdots\oplus E^{u,\sigma}_k$ be the invariant splitting for $\sigma\in\{f,g\}$. For any invariant measure $\mu$ of $\sigma\in\{f,g\}$, we denote by $\lambda^s_{l}(\mu,\sigma)<\lambda^s_{l-1}(\mu,\sigma)<\cdots<\lambda^s_1(\mu,\sigma)<0<\lambda^u_1(\mu,\sigma)<\cdots<\lambda^u_k(\mu,\sigma)$ the Lyapunov exponents of $\mu$ and $\sigma\in\{f,g\}$.

It is well known that the invariant bundles $E^{s,\sigma}_i$ and $E^{u,\sigma}_j$ are uniquely integrable and tangent to $\sigma$-invariant foliation $\cF^{s,\sigma}_i$ and $\cF^{u,\sigma}_j$ for $\sigma\in\{f,g\}$ (see for instance \cite{G08}[Lemma 6.1]). We denote by $\Gibb^{u,i}(\sigma)$ the space of Gibbs expanding states for the map $\sigma$ and the foliation $\cF^{u,\sigma}_i$ and $\Gibb^{s,j}(\sigma)$ the Gibbs expanding states for the map $\sigma^{-1}$ and the foliation $\cF^{s,\sigma}_i$.

\begin{maincor}\label{maincor.generalconjugation}

Suppose that the $C^2$ Anosov diffeomorphisms $f$ and $g$ are as above: sufficiently close to an irreducible linear Anosov diffeomorphism over the torus $\TT^d$ with $d$ real eigenvalues of different norm, such that the one-dimensional invariant sub-bundles are preserved. Suppose that the orbit of every leaf of $\cF^{s,f}_i$ and every leaf of $\cF^{u,f}_j$ are dense, for $i=1,\cdots, l$ and $j=1,\cdots, k$. Then the conjugacy $h$ between $f$ and $g$ is $C^{1+\vep}$ if and only if the following happens: for any $1\leq i \leq k$ and for any ergodic measure $\mu\in \Gibb^{u,i}(f)$, we have that $\lambda^u_i(\mu,f)=\lambda^u_i(h_*(\mu),g)$, and for any $1\leq j \leq l$ and for any ergodic measure $\mu\in \Gibb^{s,j}(f)$, we have that $\lambda^s_j(\mu,f)=\lambda^s_j(h_*(\mu),g)$.

\end{maincor}

\section{Three dimensional Anosov (Proof of Theorem \ref{main.3d})}\label{s.3danosov}

In this section, we are assuming Theorem~\ref{main.technique} and we will provide the proof of Theorem~\ref{main.3d}. We need the Journ\'{e} regularity lemma from \cite{J} for maps:

\begin{lemma}\label{l.Jour}

Let $M_j$ be a smooth manifold and $\cF^j_1$ and $\cF^j_2$, $j\in\{1,2\}$, be transverse foliations of the manifold $M_j$ whose leaves are uniformly $C^{n+\gamma}$ $n\in\mathbb N$, $n\geq 1$ and $\gamma\geq 0$. Suppose that $h: M_1\to M_2$ is a homeomorphism that maps $\cF^1_1$ into $\cF^2_1$ and $\cF^1_2$ to $\cF^2_2$. Moreover assume that the restriction of $h$ to the leaves of these foliations are uniformly $C^{n+\gamma}$, then $h$ is $C^{n+\gamma}$ if $\gamma>0$, otherwise, $h$ is $C^{n-\vep}$ for any $\vep>0$.

\end{lemma}

\begin{proof}[Proof of Theorem~\ref{main.3d}:]

We divide the proof into xxx steps.

\subsection{Basic topological facts}

By hypothesis, $L$ has three different eigenvalues, and without restriction, we can assume that they are all positive:
$0<k_1<1<k_2<k_3$. Also $L$ admits a splitting
$$
T\TT^3=E^s\oplus E^{wu}\oplus E^{su}.
$$
By our assumption, $L$ preserves the orientation of $E^i$ ($i\in\{s,wu,su\}$).

We denote by $\cA(L)$ the set of partially hyperbolic diffeomorphisms in the same isotopy class with $L$, which are called \emph{diffeomorphisms derived from Anosov}. By Franks \cite{Fra70}, for any $f\in \cA(L)$, there is a unique semi-conjugacy $h_f$ between $f$ and $L$ which is isotopic to the identity map. We need the following standard topological classification of derived from Anosov diffeomorphism by \cite{BBI09, Ham13,HaP,Pot15,Ure12}:

\begin{lemma}[\cite{YY} Proposition 2.1]\label{l.DA}

Let $f\in \cA(L)$, then $f$ is dynamically coherent, the Franks' semi-conjugation $h_f$ maps the center stable, center, center unstable and stable leaves of $f$ into the corresponding leaves of $L$.

\end{lemma}

In particular, if $f$ is an Anosov diffeomorphism, by Franks~\cite{Fra70}, the semi-conjugacy $h_f$ is indeed a conjugacy.
Denote by $\cF^s_f$, $\cF^{wu}_f$, $\cF^{su}_f$ the stable, weak-unstable and strong unstable foliation of $f$.
Then by the above lemma,

\begin{equation}\label{eq.stableandcenter}
h_f(\cF^{i}_f)=\cF^i_L \text{ for $i=s,wu$ }.
\end{equation}

\subsection{$h_{f*}\Leb=\Leb$ and $h_f\mid \cF^{s}_f$ is uniformly $C^r$}

By the hypotheses, the stable exponent of $L$ and $f$ coincide. Moreover, the support of the volume measure coincides with the whole $\TT^3$, and since $f$ preserves the volume and is Anosov, we have that the volume is ergodic and is a Gibbs $s$-state. Then as an application of Theorem~\ref{main.technique} (for $f^{-1}$ and $L^{-1}$), we get that $(h_f)_*\Leb$ is a Gibbs $s$-state of $L$, so it is the volume, and $h_f\mid \cF^{s}_f$ is uniformly $C^r$ for any $r>1$.

In fact it is easy to see from another general argument that $h_f$ preserves the volume. By Pesin formula, the metric entropy of $f$ with respect to the volume is equal to $-\log k_1$, and this must be equal to the metric entropy of $L$ with respect to $h_{f*}\Leb$. But this entropy is maximal, and since the measure of maximal entropy is unique, we have that $h_{f*}\Leb=\Leb$. So now alternatively one can apply Theorem~\ref{main.technique} and get that $h_f^{-1}\mid \cF^{s}_L$ is uniformly $C^r$.

\subsection{$h_f^{-1}\mid \cF^{wu}_L$ is uniformly $C^{1+\vep}$}

In the next step we are going to show that $h_f^{-1}\mid \cF^{wu}_L$ is uniformly $C^{1+\vep}$. Let $\vep>0$ such that $\cF^{wu}_f$ has uniformly $C^{1+\vep}$ leaves, such an $\vep$ exists because the center bundle is H\H older. From the definition, it is clear that $\Leb$ is a Gibbs $e$-state of $L$ for the expanding foliation $\cF^{wu}_L$, and whose support is the whole manifold. Recall that by Lemma~\ref{l.DA}, $h_f^{-1}(\cF^{wu}_L)=\cF^{wu}_f$, $h_{f*}^{-1}\Leb=\Leb$, and by the hypothesis, the center exponents of $L$ and $f$ are the same. Hence as an application of Theorem~\ref{main.technique}, we obtain that $h_f^{-1}\mid \cF^{wu}_L$ is uniformly $C^{1+\vep}$.

\subsection{$h_f^{-1}\mid \cF^{su}_L$ is uniformly $C^r$}

Now we need to establish the regularity of the conjugacy along the strong unstable foliation. First we need the following lemma due to Gogolev in order to show that the strong unstable foliations are preserved by the conjugacy:

\begin{lemma}[\cite{GG,GKS}]\label{l.3dstrongunstable}

Suppose $h_f$ is uniformly smooth along $\cF^{c}_f$, then $h_f(\cF^{su}_f)=\cF^{su}_L$.

\end{lemma}

Because the strong unstable exponents for $f$ and $L$ are the same, and the volume is an ergodic Gibbs $u$-sate, applying again the Theorem~\ref{main.technique} for the strong unstable foliations and $h_f^{-1}$, we obtain that $h_f^{-1}\mid \cF^{su}_L$ is uniformly $C^r$ for any $r>1$.

\subsection{Regularity of $h_f$}

Now we apply the Journ\'e Lemma twice, first we apply the lemma on every unstable leaf of $L$ for the pair of foliations $\cF^{wu}_L$ and $\cF^{su}_L$ and conclude that $h_f^{-1}\mid_{\cF^u_L}$ is uniformly $C^{1+\vep}$. Then we apply the lemma to the stable and unstable foliation of $L$ to show that $h_f^{-1}$ (and consequently $h_f$) is $C^{1+\vep}$.

Finally, by the bootstrapping argument of Gogolev \cite{G}, the conjugacy $h_f$ is indeed $C^\infty$. The proof is complete.

\end{proof}

\section{Higher dimensional Anosov (Proof of Theorem \ref{main.higherdimension} and Corollaries \ref{maincor.ECvolume}, \ref{maincor.Katokexample})\label{s.highdAnosov}}

In this Section we consider Anosov diffeomorphisms in dimension larger than three and prove Theorem~\ref{main.higherdimension}, Corollary~\ref{maincor.Katokexample} and Corollary~\ref{maincor.generalconjugation}, assuming the Theorem~\ref{main.technique}.

\subsection{Standard hypothesis}\label{standhyp}

In this subsection, we give the notations and the hypothesis which will be used throughout this section.

Let $L$ be an irreducible linear Anosov diffeomorphism on the torus $\TT^d$ with $d$ real eigenvalues of different absolute value. Denote by $\lambda^s_{l}<\lambda^s_{l-1}<\cdots<\lambda^s_1<0<\lambda^u<\cdots<\lambda^u_k$ the exponents of $L$ and by $E^s_{l}\oplus E^s_{l-1}\oplus \cdots\oplus E^s_1\oplus E^u_1\oplus \cdots\oplus E^u_k$ the invariant splitting corresponding to the eigenvectors. For a $C^2$ volume preserving diffeomorphism $f$, belonging to a small enough $C^1$ neighborhood $\cU$ of $L$, the dominated splitting from above will persist. Denote by $E^{s,f}_{l}\oplus E^{s,f}_{l-1}\oplus \cdots\oplus E^{s,f}_1\oplus E^{u,f}_1\oplus \cdots\oplus E^{u,f}_k$ the corresponding invariant splitting, and for any invariant measure $\mu$ of $f$, denote by $\lambda^s_{l}(\mu,f)<\lambda^s_{l-1}(\mu,f)<\cdots<\lambda^s_1(\mu,f)<0<\lambda^u_1(\mu,f)<\cdots<\lambda^u_k(\mu,f)$ the exponents of $f$ with respect to $\mu$.

\begin{lemma}[\cite{G08} Lemma 6.2]\label{l.1dintegrable}

If the neighborhood $\cU$ is sufficiently small, then the invariant bundles $E^{s,f}_i$ and $E^{u,f}_j$ are uniquely integrable and tangent to $f$-invariant foliation $\cF^{s,f}_i$ and $\cF^{u,f}_j$ with uniform $C^{1+\vep}$ leaves for some $\vep>0$, for all $1\leq i\leq l$ and $1\leq j\leq k$.

\end{lemma}

We denote by $\Gibb^{u,i}(f)$ the space of Gibbs expanding states for $f$ and the expanding foliation $\cF^{u,f}_i$ and $\Gibb^{s,j}(f)$ the Gibbs expanding states for the map $f^{-1}$ and the foliation $\cF^{s,f}_i$.

Write $E^{s,f}_{1,i}=E^{s,f}_{i}\oplus \cdots\oplus E^{s,f}_1$ and $E^{u,f}_{1,j}=E^{u,f}_{1}\oplus \cdots\oplus E^{u,f}_j$, then, by \cite{HPS77}, these bundles are uniquely integrable (for a proof see also \cite{G08}[Lemmas 6.1]). Denote by $\cF^{s,f}_{1,i}$ and $\cF^{u,f}_{1,j}$ the corresponding integral foliations. The following is a direct corollary of \cite{HPS77}.

\begin{lemma}[\cite{G08} Lemma 6.3]\label{l.centerintegrable}

For two Anosov diffeomorphisms $f,g\in \cU$, denote by $h$ the conjugacy between $f$ and $g$. Then for any $1\leq i\leq l$ and $1\leq j \leq k$, we have that

$$
h(\cF^{s,f}_{1,i})=\cF^{s,g}_{1,i} \text{ and } h(\cF^{u,f}_{1,j})=\cF^{u,g}_{1,j}.
$$

\end{lemma}

A non-trivial fact which was shown by Gogolev is that, with some additional hypothesis, the conjugacy also preserves intermediate foliations.

\begin{lemma}[Gogolev \cite{G08} Lemma 6.6]\label{l.inductioncenterleaf}

For two Anosov diffeomorphisms $f,g\in \cU$, denote by $h$ the conjugacy between $f$ and $g$. Let $1\leq j\leq k-1$, and suppose that $h\mid_{\cF^{u,f}_i}$, ${i=1,\cdots j}$ are uniformly $C^{1+\vep}$ and $h({\cF^{u,f}_i})=\cF^{u,g}_i$ ${i=1,\cdots j}$, then $h_f(\cF^{u,f}_{j+1})=\cF^{u,g}_{j+1}$.

\end{lemma}

We will write $\cF^{u,f}=\cF^{u,f}_{1,k}$ and $\cF^{s,f}=\cF^{s,f}_{1,l}$.

\subsection{Proof of Theorem~\ref{main.higherdimension}}

\begin{proof}[Proof of Theorem~\ref{main.higherdimension}:]

Denote by $h_f$ the conjugacy between $f$ and $L$, meaning that $L\circ h_f=h_f\circ f$. We will prove first that $h_f\mid_{\cF^{u,f}}$ is uniform $C^{1+\vep}$. By Journ\'{e}'s regularity lemma, it is sufficient for us to show that $h_f\mid_{\cF^{u,f}_i}$ is uniformly $C^{1+\vep}$.

First we remark that since the metric entropy of $f$ with respect to the volume is equal to the metric entropy of $L$ with respect to $h_{f*}\Leb$, and by Pesin formula is maximal, we obtain again that $h$ preserves the volume. We also remark that the volume is a Gibbs expanding measure for each $\cF^{u,L}_i$, and of course it has full support.

We prove the result by induction. For $i=1$, because $h_f(\cF^{u,f}_1)=\cF^{u,L}_1$ (Lemma \ref{l.centerintegrable}), and by our hypothesis, $\lambda^{u}_1(\Leb,f)=\lambda^u_1(h_*(\Leb),L)=\lambda^u_1$, then by (B4$^\prime$) of Theorem~\ref{main.technique}, we have that $h_f\mid_{\cF^{u,f}_1}$ is uniformly $C^{1+\vep}$ for some $\vep>0$.

Now suppose that for some $1\leq j \leq l-1$ we have that $h_f\mid_{\cF^{u,f}_i}$, for ${i=1,\cdots j}$, are uniformly $C^{1+\vep}$, and that $h_f({\cF^{u,f}_i})=\cF^{u,L}_i$, for ${i=1,\cdots j}$ . Then by Lemma~\ref{l.inductioncenterleaf}, we obtain that $h_f(\cF^{u,L}_{j+1})=\cF^{u,f}_{j+1}$. Then repeating the above argument and using again Theorem \ref{main.technique}, we can show that $h\mid_{\cF^{u,f}_{j+1}}$ is uniformly $C^{1+\vep}$.

Since $h_f$ is uniformly smooth along the one dimensional expanding foliations, we conclude that $h_f\mid_{\cF^{u,f}}$ is uniformly $C^{1+\vep}$.

By a similar argument, we can show that $h_f\mid_{\cF^{s,f}}$ is also uniformly $C^{1+\vep}$. Applying the Journ\'{e} regularity lemma again, we conclude that $h_f$ is $C^{1+\vep}$.

\end{proof}

\subsection{Proof of Corollary \ref{maincor.ECvolume}}

\begin{proof}[Proof of Corollary \ref{maincor.ECvolume}:]

Assume that the first hypothesis on the volume is satisfied. Observe that by Lemma \ref{l.1dintegrable} and Lemma \ref{l.centerintegrable} we know that the unstable foliation of $f$ decomposes into one dimensional unstable foliations $\cF^{u,f}_i$, which have uniformly $C^{1+\vep}$ leaves, since all the sub-bundles are H\H older continuous. We know by hypothesis that $h_{\Leb}(f,\cF^{u,f}_i)=\lambda^u_i$, for all $1\leq i\leq k$. By Ruelle inequality, the corresponding Lyapunov exponent of $f$ for $\Leb$ and $\cF^{u,f}_i$ must satisfy
$$
h_{\Leb}(f,\cF^{u,f}_i)=\lambda^u_i\leq \lambda^u_i(f,\Leb),\ \ \forall 1\leq i\leq k.
$$
If one of these inequalities is strict, then taking the sum for all $i$ we obtain
\begin{equation}\label{eq.Lyapunovsum}
\lambda^u(f,\Leb)=\sum_{i=1}^k\lambda^u_i(f,\Leb)>\sum_{i=1}^k\lambda^u_i=\lambda^u(L),
\end{equation}
which is a contradiction, because the absolute continuity of $\cF^{u,f}$ and the so-called Ma\~ ne argument would imply that the volume of the unstable leaves grows faster than the topological growth of the linear part, which is impossible (see \cite{SX09} or \cite{Sa}).

An alternative argument is the following: By Pesin formula we have that $\lambda^u(f,\Leb)=h_{\Leb}(f)$; since $f$ and $L$ are conjugated, we have $\lambda^u(L)=h_{top}(L)=h_{top}(f)$; the Variational principle says that $h_{\Leb}(f)\leq h_{top}(f)$. This three facts together give that $\lambda^u(f,\Leb)=h_{\Leb}(f)\leq h_{top}(f)=\lambda^u(L)$, which contradicts \eqref{eq.Lyapunovsum}.

Consequently we have that $\lambda^u_i(f,\Leb)= h_{\Leb}(f,\cF^{u,f}_i)=\lambda^u_i,\ \ \forall 1\leq i\leq k$. The same argument works for the corresponding one dimensional stable foliations and for $f^{-1}$. This reduces the proof to Theorem \ref{main.higherdimension}.

Now assume that the second condition is satisfied. This means that for any $\mu_i$, ergodic Gibbs $e$-state for the foliation $\cF^{u,f}_i$, we have $\lambda^u_i(f,\mu_i)=\lambda^u_i$ (and a similar condition for the stable foliation holds). Since the Lyapunov exponents of $L$ are the same for any invariant measure, the condition (B4) of Theorem \ref{main.technique} is satisfied for $(f,\cF^{u,f}_1,\mu_1)$ and $(L,\cF^{u,L}_1,h_*\mu_1)$, so we obtain that condition (B5') is also satisfied so $h$ is $C^{1+\delta}$ along $\cF^{u,f}_1$ on the support of $\mu_1$. Also $h_*\mu_1$ is absolutely continuous along $\cF^{u,L}_1$, and since $L$ is irreducible we get that $h_*\mu_1$ must be the volume on $\TT^d$, so it has full support. By induction on $i\in\{1,2,\dots\}$ and using Lemma \ref{l.inductioncenterleaf} in order to obtain that $h$ preserves the one dimensional foliations, we conclude again that $h$ is $C^{1+\delta}$ along each unstable subfoliation. A similar argument for the stable foliation and the Journ\'e regularity lemma will then give us the desired conclusion.

\end{proof}

\subsection{Proof of Corollary~\ref{maincor.Katokexample}}

We need the following version of Journ\'e regularity lemma for foliations, which can be found in \cite{DX}[Lemma 4.3], and follows directly from the claims in \cite{PSW} (see also~\cite{BX}).

\begin{lemma}\label{l.Jourfoliation}

Suppose that $\cW,\cF$ and $\cL$ are foliations of the manifold $M$, $\cW$ and $\cF$ subfoliate $\cL$, and $\cW$ is transversal to $\cF$ (within $\cL$). Let $\cH=\cH^{\cW}$ be the holonomy within the leaves of $\cL$, along the leaves of $\cW$, between the leaves of $\cF$. Let $r>1$, $r\notin\mathbb N$. If the foliations $\cF,\cW,\cL$ have uniformly $C^{r}$ leaves, and $\cH$ is uniformly $C^{r}$, then $\cW$ is a $C^{r}$ foliation within $\cL$.

\end{lemma}

\begin{proof}[Proof of Corollary~\ref{maincor.Katokexample}:]

We first assume that all the exponents of $\gamma_t$ coincide with the Lyapunov exponents of $L$. To apply the Journ\'e regularity lemma for foliations, we take $\cL$ to be the whole manifold $\TT^d\times I$ (it has only one leaf), $\cF=\{\TT^d\times \{t\}\}_{t\in I}$, and $\cW$ to be the center foliation. Then, by Hirsch, Pugh and Shub, the leaves of $\cW$ are uniformly $C^{1+\vep}$ ($\cL$ and $\cF$ are $C^{\infty}$). Also $\cW$ and $\cF$ are transversal and subfoliate $\cL$.

By the invariance of the center foliation under iteration of $f$, the holonomy map $\cH^c_{t_1,t_2}$ induced by the center foliation $\cW$ between the $\cF$ leaves $\TT^d\times \{t_1\}$ and $\TT^d\times \{t_2\}$ is the conjugacy between the Anosov diffeomorphisms $\gamma_{t_1}$ and $\gamma_{t_2}$. Thus

$$
\cH^c_{t_1,t_2}=h_{\gamma_{t_2}}^{-1}\circ h_{\gamma_{t_1}},
$$
where $h_{\gamma_t}$ denotes the conjugacy between $\gamma_t$ and $L$.

As a consequence of Theorem~\ref{main.higherdimension}, $\gamma_t$ is $C^{1+\vep}$ for any $t\in I$. Thus the holonomy map $\cH^c_{t_1,t_2}$ is uniformly $C^{1+\vep}$. Then, by Journ\'e regularity lemma for foliation mentioned above, the center foliation $\cW$ is $C^{1+\vep}$.

Now suppose that if $t_1\neq t_2$, then $\gamma_{t_1}$ and $\gamma_{t_2}$ have the sums of the unstable Lyapunov exponents different. Then the conjugacy $\cH^c_{t_1,t_2}$ between $\gamma_{t_1}$ and $\gamma_{t_2}$ cannot be smooth. Furthermore $(\cH^c_{t_1,t_2})_*(\Leb)$ cannot be the volume because this would contradict the Pesin formula, so $(\cH^c_{t_1,t_2})_*(\Leb)$ must be a measure which is singular with respect to the volume on  $\TT^d\times\{t_2\}$.

We now take $\Gamma_t$ to be the set of regular points of the map $\gamma_t$ for $t\in I$. By Birkhoff theorem, $\Leb_{\TT^d\times\{ t\}}(\Gamma_t)=1$. Let $\Gamma=\cup_t \Gamma_t$, then by Fubini theorem, $\Gamma$ has full volume in the manifold $\TT^d\times I$.

We claim that $\Gamma$ intersects every center leaf in at most one point. Suppose by contradiction that there is a center leaf which intersects $\Gamma_{t_1}$ and $\Gamma_{t_2}$. Then there are two points $(x_1, t_1)$ and $(x_2,t_2)$ such that $\cH^c_{t_1,t_2}(x_1,t_1)=(x_2,t_2)$, and the orbit of $(x_1,t_1)$ is also mapped to the orbit of $(x_2,t_2)$ by the map $\cH^c_{t_1,t_2}$. Thus,
$$
(\cH^c_{t_1,t_2})_*\frac 1n(\sum_{i=0}^{n-1} \delta_{\gamma_{t_1}^i(x_1,t_1)})=\frac 1n\sum_{i=0}^{n-1} \delta_{\gamma_{t_2}^i(x_2,t_2)}.
$$
Take $n\to \infty$, and observe that $(x_1,t_1)$ is the regular point of $\gamma_{t_1}$ for the Lebesgue measure and $(x_2,t_2)$ is the regular point of $\gamma_{t_2}$ for the Lebesgue measure, we have $(\cH^c_{t_1,t_2})_*\Leb=\Leb$. This is  a contradiction, so the proof is finished.

\end{proof}

\subsection{Proof of Corollary~\ref{maincor.generalconjugation}}

The proof of Theorem~\ref{maincor.generalconjugation} is the same with the proof of Theorem~\ref{main.higherdimension}, one only needs to replace $L$ by $g$ and observe that from (A3) of Proposition~\ref{p.Gibbsestate} and from our hypothesis, we get that for any ergodic measure $\mu\in \Gibb^{u,i}(f)$ $(i=1,\cdots, k)$, and any ergodic measure $\mu^\prime \in \Gibb^{s,j}(f)$ $(j=1,\cdots l)$, $\supp(\mu)=\supp(\mu^\prime)$ coincides to the ambient manifold. So we omit the proof here.

\section{Proof of Theorem~\ref{main.technique} and Corollary~\ref{maincor.periodicpoints} \label{s.technique}}

In this section we provide the proof of the main technical result of the paper, Theorem~\ref{main.technique}, and the Corollary~\ref{maincor.periodicpoints}. We remark that the proof does not depend on Livshitz theorem, and may have potentially other applications for partially hyperbolic diffeomorphisms without (many) periodic points.

\subsection{B1 implies B2}\label{B1B2}

The idea of the proof is the following. If the conjugacy preserves 2 invariant expanding foliations and two corresponding Gibbs expanding measures, then it must also preserve the disintegrations and the quotient measures for corresponding foliation boxes. Uniform bounds for the regularity of the densities of the disintegrations would give us the desired conclusion.

\begin{proof}

Let $B^f$ be a foliation box of the foliation $\cF^f$, then $B^g=h(B^f)$ is a foliation box for the foliation $\cF^g$. Let $D^f$ be a transverse disk of the foliation chart $B^f$, then we have the local projection map $\pi^f: B^f \to D^f$ along the local $\cF^f$ leaves. The image $D^g=h(D^f)$ is a topological transverse disk of the foliation box $B^g$, and again we have the local projection map $\pi^g: B^g \to D^g$ along the local $\cF^g$ leaves.

We suppose that $\mu(B^f)>0$, then we also have $\nu(B^g)>0$. Denote by $\mu_B$ and $\nu_B$ the normalized restrictions of $\mu$ to $B^f$, respectively $\nu$ to $B^g$. We will have that $\nu_B=h_*(\mu_B)$. Then we can write the disintegration of $\mu_B$ along the foliation $\cF^f$ by:

\begin{equation}\label{eq.disintegration}
d\mu_B(y)=\int_{D^f} d\mu_{x}(y) d\overline{\mu}(x)
\end{equation}
where $\mu_{x}$ is the disintegration of $\mu_B$ along the plaque $\cF^f(x)$ for $\overline{\mu}$ almost every $x\in D^f$, and $\overline{\mu}=\pi^f_{*}\mu_B$ is the quotient measure.

We can also write the disintegration of $\nu_B$ along the foliation $\cF^g$ by:

$$
d\nu_B(y)=\int_{D^g} d\nu_{x}(y) d\overline{\nu}(x)
$$
where $\nu_{x}$ is the disintegration of $\nu_B$ along the plaque $\cF^g(x)$ for $\overline{\nu}$ almost every $x\in D^g$, and $\overline{\nu}=\pi^g_{*}\nu_B$ is the quotient measure.

By the uniqueness of disintegrations, we know that

\begin{lemma}\label{l.uniquenessofdisintegration}

$h_*(\overline{\mu})=\overline{\nu}$ and $h_*(\mu_{x})=\nu_{h(x)}$ for $\overline{\mu}$ almost every $x\in D^f$.

\end{lemma}

As shown by A4 of Proposition~\ref{p.Gibbsestate}, for $\overline{\mu}$ almost every $x$,

\begin{equation}\label{eq.leafmeasure}
d\mu_{x}=\rho^f_xd\Leb\mid_{\cF^f(x)},
\end{equation}
where $\rho^f_x$ is given by \eqref{eq.density}, and is continuous on $B^f$ and uniformly $C^{r-1}$ along the leaves. Although the conditional measure is only defined on $\mu$ almost all leaves, we can extend the definition of $\mu_{x}$ to the support of $\overline{\mu}$ using the uniform continuity of the density $\rho^f$. The previous discussion is equivalent to say that for {\it every} point $x\in supp(\mu\cap B^f)$, $\mu_{x}$ is well-defined on the leaf $\cF^f(x)$, and satisfies equation \eqref{eq.leafmeasure}, where $\rho^f_x$ is given by \eqref{eq.density}, and is uniformly $C^{r-1}$ along the leaves.

Since we assume $\nu\in \Gibb^e(g,\cF^g)$, then by a similar argument as above, for {\it every} point $x\in supp(\nu)\cap B^g$, $\nu_{x}$ is well-defined on the leaf $\cF^g(x)$, and

$$
d\nu_{x}=\rho^g_xd\Leb\mid_{\cF^g(x)},
$$
where $\rho^g_x$ is given by \eqref{eq.density}, and is uniformly $C^{r-1}$ along the leaves.

We saw that $h_*(\mu_{x})=\nu_{h(x)}$ holds for $\overline{\mu}$ almost every $x\in D^f$, or equivalently $\mu$ almost every $x\in B^f$. Since $\rho^f_x$ and $\rho^g_y$ are continuous with respect to $x\in B^f$ and $y\in B^g$, and the plaques of $\cF^f(x)$ and $\cF^g(x)$ are uniformly $C^r$ and vary continuously in the $C^1$ topology with respect to $x$ (meaning that the plaques are images of $C^1$ functions $\sigma_x:B(0,1)\subset\mathbb R^{\dim(\cF)}\rightarrow M$, and $x\mapsto\sigma_x$ is continuous for the $C^1$ topology on the right-hand side), we get that $\mu_x$ and $\nu_x$ are continuous with respect to the point $x$. Because $h$ is continuous, we conclude that $h_*(\mu_{x})=\nu_{h(x)}$ for {\it every} point $x\in B^f\cap\supp\mu$.

For any $x\in B^f\cap\supp\mu$, we claim that $h\mid_{\cF^f(x)}$ is absolutely continuous, or $h_*(\Leb\mid_{\cF^f(x)})$ and $\Leb\mid_{\cF^g(h(x))}$ are equivalent. This follows from the following equivalences, using that $h_*(\mu_{x})=\nu_{h(x)}$, $\mu_x$ is equivalent to $\Leb\mid_{\cF^f(x)}$, and $\nu_{h(x)}$ is equivalent to $\Leb\mid_{\cF^g(h(x))}$:
\begin{eqnarray*}
\Leb\mid_{\cF^g(h(x))}(A) =0&\iff& \nu_{h(x)}(A)=0\iff h_*(\mu_{x})(A)=0\\
&\iff& \mu_x(h^{-1}(A))=0\iff \Leb\mid_{\cF^f(x)}(h^{-1}(A))=0
\end{eqnarray*} 
We will compute the Jacobian of $h\mid_{\cF^f(x)}$, or the Radon-Nicodym derivative of $\Leb\mid_{\cF^g(h(x))}$ with respect to $h_*(\Leb\mid_{\cF^f(x)})$. Using that $d\Leb\mid_{\cF^g(h(x))}=\frac 1{\rho_{h(x)}^g}d\nu_{h(x)}$, $h_*(\mu_{x})=\nu_{h(x)}$, and $d\mu_{x}=\rho^f_xd\Leb\mid_{\cF^f(x)}$, we obtain:
\begin{eqnarray*}
d\Leb\mid_{\cF^g(h(x))}&=& \frac 1{\rho^g_{h(x)}}d\nu_{h(x)}\\
&=& \frac 1{\rho^g_{h(x)}}dh_*\mu_{x}\\
&=& \frac 1{\rho^g_{h(x)}}h_*\left(\rho^f_xd\Leb\mid_{\cF^f(x)}\right)\\
&=& \frac {\rho^f_x\circ h^{-1}}{\rho^g_{h(x)}}dh_*\left(\Leb\mid_{\cF^f(x)}\right).
\end{eqnarray*} 
In conclusion we have that the Jacobian of $h\mid_{\cF^f(x)}$ is
\begin{equation}\label{eq.jacobian}
J(h\mid_{\cF^f(x)})(z)=\frac{d\Leb\mid_{\cF^g(h(x))}}{dh_*\left(\Leb\mid_{\cF^f(x)}\right)}(h(z))=\frac {\rho^f_x(z)}{\rho^g_{h(x)}(h(z))},
\end{equation}
for every $x\in B^f\cap\supp\mu$ and every $z$ in the plaque $\cF^f(x)$. By the continuity of $\rho^f,\rho^g$ and $h$ we see that $J(h\mid_{\cF^f})$ is also continuous on $B^f\cap\supp\mu$, and since $J(h\mid_{\cF^f})$ is independent of the election of the foliation box $B^f$, we get that $J(h\mid_{\cF^f})$ is continuous on $\supp\mu$. Since the support of $\mu$ is compact, $J(h\mid_{\cF^f})$ must be bounded away from zero and infinity.

\end{proof}

\begin{remark}\label{rk.jacobian}

We remark that we obtained that the map $h\mid_{\cF^f(x)}$ is absolutely continuous, and we will see that it is in fact differentiable if the dimension of $\cF$ is 1. However $h\mid_{\cF^f(x)}$ may not be differentiable if the dimension of $\cF$ is strictly greater than 1, so in this case the Jacobian means in fact the Radon-Nicodym derivative mentioned above.

\end{remark}

\subsection{(B2) implies (B3)}

It is enough to observe that $h\circ f^n=g^n\circ h$ is absolutely continuos when restricted to leaves of $\cF^f$ in the support of $\mu$, as a composition of an absolutely continuous function and a differentiable function. Computing the Jacobian along $\cF^f$ for some $x\in\supp\mu$ we get
$$
J(h\mid_{\cF^f})(f^n(x))\cdot\det (Df^n\mid_{T_x\cF^f(x)})=\det (Dg^n\mid_{T_{h(x)}\cF^g(h(x))})\cdot J(h\mid_{\cF^f})(x),
$$
or
$$
\frac{\det (Df^n\mid_{T_x\cF^f(x)})}{\det (Dg^n\mid_{T_{h(x)}\cF^g(h(x))})}=\frac{ J(h\mid_{\cF^f})(x)}{J(h\mid_{\cF^f})(f^n(x))}.
$$

Since $J(h\mid_{\cF^f})$ is uniformly bounded away from zero and infinity on the support of $\mu$, the conclusion follows.

\subsection{(B3) implies (B4)}

This is a direct consequence of Oseledets theorem, Birkhoff ergodic theorem, and the fact that the logarithm of the determinant restricted to the tangent bundle of the foliation is an additive cocycle. It is sufficient to prove that
$$
\int \log (\det(Df\mid_{T_x\cF^f}))d\mu(x)=\int \log (\det(Dg\mid_{T_x\cF^g}))d\nu(x).
$$
Since

$$
\int \log (\det(Df^n\mid_{T_x\cF^f}))d\mu(x)=n\int \log (\det(Df\mid_{T_x\cF^f}))d\mu(x)
$$
we get
\begin{equation}
\begin{aligned}
\int \log (\det(Df\mid_{T_x\cF^f}))d\mu(x)&=&\lim_{n\to \infty} \frac{1}{n}\int \log (\det(Df^n\mid_{T_x\cF^f}))d\mu(x)\\
&=&\lim_{n\to \infty} \frac{1}{n}\int \log (\det(Dg^n\mid_{T_{h(x)}\cF^g}))d\mu(x)\\
&=&\lim_{n\to \infty} \frac{1}{n}\int \log (\det(Dg^n\mid_{T_{y}\cF^g}))d\nu(y)
\end{aligned}
\end{equation}

In the above relations we used the hypothesis (B3) to conclude that the two limits are equal, and the fact that $\nu=h_*(\mu)$ in order to change the variable.

The proof of (B3)$\Longrightarrow$ (B4) is complete.

\subsection{(B4) implies (B1)}

\begin{proof}

Let $\xi^f$ be an subordinated partition of $(\cF^f,f,\mu)$. Because $h_*(\mu)=\nu$ and $h(\cF^f)=\cF^g$, we obtain that $h(\xi^f)$ is also a subordinated partition of $(\cF^g,g,\nu)$. Moreover, $h_\mu(f,\cF^f)=h_\nu(g,\cF^g)$.

Because $\mu\in \Gibb^e(f,\cF^f)$, by Theorem~\ref{l.criterionofgibbs} we have that $h_\mu(f,\cF^f)=\lambda^{\cF^f}(f,\mu)$. Also by the hypothesis of (B4), $\lambda^{\cF^f}(f,\mu)=\lambda^{\cF^g}(g,\nu)$. Thus, we have
$$
h_\nu(g,\cF^g)=\lambda^{\cF^g}(g,\nu).
$$
By Theorem~\ref{l.criterionofgibbs} again, $\nu\in \Gibb^e(g,\cF^g)$. The proof is complete.
\end{proof}

\subsection{1-dimensional expanding foliation}

In this subsection, we assume that $\dim(\cF^f)=\dim(\cF^g)=1$. We are going to show that (B5$^\prime$) is equivalent to (B1), (B2), (B3) and (B4). Indeed, we are going to show (B5$^\prime$) implies (B2), and (B1) implies (B5$^\prime$).

\begin{proof}[Proof of (B5$^\prime$) implies (B2):]
This implication is straightforward, since $C^r$ regularity implies absolute continuity with bounded Jacobian. This implication holds in any dimension.

\end{proof}

\begin{proof}[Proof of (B1) implies (B5$^\prime$):]
Recall from the subsection \ref{B1B2} that we have the support of the measure $\mu$ covered by finitely many foliation boxes $B_i^f$, $i\in\{1,\dots k\}$, and the disintegration of the measure $\mu_i=\mu\mid_{B_i^f}$ is
$$
d\mu_i(y)=\int_{D^f_i} d\mu_{i,x}(y) d\overline{\mu}_i(x),
$$
where $\mu_{i,x}(z)=\rho^f_x(z)d\Leb\mid_{\cF^f_i(x)}(z)$, for {\it every} $x\in B^f_i$, while $\rho^f$ is uniformly $C^{r-1}$ along the leaves. Also the disintegration of the measure $\nu_i=\nu\mid_{B_i^g}$ is
$$
d\nu_i(y)=\int_{D^g_i} d\nu_{i,x}(y) d\overline{\nu}_i(x),
$$
where $\nu_{i,x}(z)=\rho^g_x(z)d\Leb\mid_{\cF^g_i(x)}(z)$, for {\it every} $x\in B^g_i$, while $\rho^g$ is uniformly $C^{r-1}$ along the leaves.

Lemma \ref{l.uniquenessofdisintegration} tells us that $h_*(\mu_{i,x})=\nu_{i,h(x)}$ for $\overline{\mu}_i$ almost every $x$, and again by continuity it will hold for {\it every} $x\in\supp(\mu_i)$.

Recall that we obtained in subsection \ref{B1B2}, equation \eqref{eq.jacobian} that the Jacobian of $h\mid_{\cF^f(x)}$ is
$$
J(h\mid_{\cF^f(x)})(z)=\frac {\rho^f_x(z)}{\rho^g_{h(x)}(h(z))}.
$$

The following is a well-known argument used to obtain the regularity along one dimensional foliations, and it was used extensively in dynamics. If the leaves of $\cF$ are one-dimensional, then the Lebesgue measure on the plaques of $\cF^f$ and $\cF^g$ is exactly the arc length, and the Jacobian of $h\mid_{\cF^f(x)}$ is exactly the derivative of $h\mid_{\cF^f(x)}$, when both $\cF^f(x)$ and $\cF^g(h(x))$ are parametrized by the arc length, i.e.
$$
\left(h\mid_{\cF^f(x)}\right)'(z)=\frac {\rho^f_x(z)}{\rho^g_{h(x)}(h\mid_{\cF^f(x)}(z))}.
$$
Because the leaves and the densities $\rho^f_x$ and $\rho^g_{h(x)}$ are uniformly $C^{r-1}$, a standard induction argument implies that $h\mid_{\cF^f(x)}$ is uniformly $C^r$.

Since the argument works for all $i\in\{1,\dots k\}$, the proof is finished.

\end{proof}

\begin{remark}\label{rk.1D}

We remark that the assumption that the dimension of the foliations is one is fundamental for the last step of the proof. Without this assumption the implication ``(B1)$\Rightarrow$(B5')'' may not be true.

\end{remark}

\subsection{Proof of Corollary~\ref{maincor.periodicpoints}}

First suppose that the hypothesis of Theorem~\ref{main.technique} is satisfied, then by (B3) of Theorem~\ref{main.technique}, there exists $K>0$ such that for every periodic point $p\in \supp(\mu)$ with period $\pi(p)$, we have
$$
\frac{1}{K}<\frac{\det(Df^{n\pi(p)}\mid_{T_p\cF^f})}{\det(Dg^{n\pi(p)}\mid_{T_{h(p)}\cF^g})}<K.
$$
Thus
$$
\frac{1}{\sqrt[n]{K}}<\frac{\det(Df^{\pi(p)}\mid_{T_p\cF^f})}{\det(Dg^{\pi(p)}\mid_{T_{h(p)}\cF^g})}<\sqrt[n]{K}.
$$
Letting $n\to \infty$, we conclude that
$$
\det(Df^{\pi(p)}\mid_{T_p\cF^f})=\det(Dg^{\pi(p)}\mid_{T_{h(p)}\cF^g}).
$$

Conversely, suppose now that the above equality holds for a sequence of periodic points $p_n$, which are not necessary contained in $\supp(\mu)$, and $\mu$ belongs to the closure of the convex set generated by $\mu_n$, where $\mu_n=\sum_{i=0}^{\pi(p_n)-1}\delta_{f^i(p_n)}$.

Write $\nu_n=\sum_{i=0}^{\pi(p_n)-1}\delta_{g^i(h(p_n))}$, then $h_*(\mu_n)=\nu_n$. Because $h_*(\mu)=\nu$, then $\nu$ belongs to the closure of the convex set generated by $\nu_n$.

By the assumption, we have that for every $n\in\mathbb N$
$$
\int \log(\det(Df\mid_{T\cF^f_x}))d\mu_n(x)=\int \log(\det(Dg\mid_{T\cF^g_x}))d\nu_n(x).
$$
From the linearity and the continuity of the integral of a continuous function with respect to the measure, we obtain

$$
\int \log(\det(Df\mid_{T\cF^f_x}))d\mu(x)=\int \log(\det(Dg\mid_{T\cF^g_x}))d\nu(x).
$$

Then the condition (B4) of Theorem~\ref{main.technique} is satisfied, and consequently we have that $\nu\in \Gibb^e(g,\cF^g)$. The proof is finished.

\section{A version of the invariance principle\label{s.invarianceprinciple}}

In this section we will introduce a version of the invariance principle, which is a relativized variation principle for partial entropy along expanding foliations, developed by Tahzibi-Yang \cite{TY} and Viana-Yang \cite{VY3}. There are many works on various versions of invariant principles, for more details see \cite{Led,AV,ASV,AVW1,AVW2,TY}. In general one assumes that the central Lyapunov exponents vanish, and obtains that the disintegrations of the measure along the central leaves are invariant under the stable and unstable holonomies, and under 'good' conditions one even gets that the disintegrations along the center leaves are also continuous.

We are interested here about the invariance of the disintegrations along (weak) stable and unstable leaves, under the holonomy along the center foliation. Throughout this section, let $M$ and $\overline M$ be smooth Riemannian manifolds, $M$ being a continuous fiber bundle over $\overline M$ with compact fiber, $f\in \Diff^1(M)$ be a diffeomorphism which fibers over $\overline f\in\Diff^1(\overline M)$. If $\pi:M\rightarrow\overline M$ is the fiber projection, this says that $\pi$ is a semiconjugacy between $f$ and $\overline f$, $h\circ f=\overline f\circ h$, and $\pi^{-1}(y)$ is homeomorphic to the compact fiber for every $y\in\overline M$. Assume that there exist $\cF$ and $\overline\cF$ expanding foliations with $C^1$ leaves for $f$ respectively $\overline f$, such that $\pi(\cF)=\overline\cF$, and $\pi\mid_{\cF(x)}$ is a homeomorphism between $\cF(x)$ and $\overline\cF(\pi(x))$, for every $x\in M$.

Under this condition one can define the 'center holonomy', or the holonomy along the fibers, inside $\pi^{-1}(\overline\cF(z))$. If $x,y\in\pi^{-1}(\overline\cF(z))$ we define the center holonomy between $\cF(x)$ and $\cF(y)$ to be $\cH^c_{x,y}=\left(\pi\mid_{\cF(y)}\right)^{-1}\circ\pi\mid_{\cF(x)}$, i.e. $v$ and $\cH^c_{x,y}(v)$ belong to the same fiber.

Let $\nu$ be an invariant measure for $\overline f$, and $\mu$ an invariant measure (not necessarily ergodic) for $f$ such that $\pi_*(\mu)=\nu$. Given $\overline \xi$ a subordinated partition for the expanding foliation $\overline\cF$ for $(\overline f,\nu)$ with small enough pieces, it can be lifted to a subordinated partition $\xi$ for the expanding foliation $\cF$ for $(f,\mu)$, by taking the intersection of local leaves of $\cF$ with preimages under $\pi$ of pieces of $\overline\xi$. Let $\mu_x$ and $\nu_y$ be the conditional measures of $\mu$ respectively $\nu$, along the partitions $\xi$ respectively $\overline\xi$. Now we will define the $c$-invariance of $\mu$.

\begin{definition}

We say that {\it $\mu$ is $c$-invariant for $\cF$} if for any subordinated partition $\overline\xi$ of $\overline\cF$, and the induced subordinated partition $\xi$ of $\cF$, the center holonomy preserves the conditional measures of $\mu$ along $\xi$, i.e. there exists a full measure set $\Gamma^1\subset M$ such that for every $x,y\in\Gamma^1$ with $y\in\pi^{-1}(\overline\cF(\pi(x)))$, we have $\cH^c_{x,y*}\mu_x=\mu_y$.

Equivalently, there exists a full measure set $\Gamma\subset M$ such that for any $x\in\Gamma$,

\begin{equation}\label{eq.cinvariance}
\pi_*\mu_x=\nu_{\pi(x)}.
\end{equation}

\end{definition}

Our main tool is the following criterion which was used in Tahzibi-Yang \cite{TY} and Viana-Yang \cite{VY3} in order to obtain the $c$-invariance.

\begin{theorem}\label{p.partialentropy}

Let $\mu$ be an invariant measure of $f$, and $\nu=\pi_*(\mu)$ be an invariant measure of $\overline{f}$. Then
$$
h_\mu(f,\cF^u)\leq h_\nu(\overline{f},\overline\cF),
$$
and the equality holds if and only if $\mu$ is $c$ invariant for $\cF$.

\end{theorem}

The authors used the $c$-invariance in order to obtain $s$ and $u$- invariance, and thus this is similar in some sense with the traditional invariance principle. However the center invariance is enough in our considerations.

In fact the exact statement in \cite{TY} [Theorem A] is more restrictive, because they deal with the unstable foliations of a partially hyperbolic diffeomorphism of skew-product type, however the proof uses only our hypothesis (see also \cite{VY3}). The ideas behind this result are very similar to the ideas behind the Pesin formula and its converse. A sketch of the proof is the following.

\begin{proof}[Proof of Theorem \ref{p.partialentropy}]

We will use the notations described above. Let $\tilde\mu_x$ be the disintegrations of $\mu$ along the partition $\pi^{-1}\overline\xi$. Then for $\mu$ almost every $x\in M$, the conditional measures of $\tilde\mu_x$ along $\xi$ are $\mu_y$, the same as the conditionals of $\mu$ along $\xi$, and denote by $\overline\mu_x$ the quotient measure. Then, for $\mu$ almost every $x$, $\pi_*\tilde\mu_x=\nu_{\pi(x)}$, and for every measurable set $A\in\overline\xi(\pi(x))$ we have
$$
\nu_{\pi(x)}(A)=\tilde\mu_x(\pi^{-1}(A))=\int\mu_z(\pi^{-1}(A))d\overline\mu_x(z).
$$

This says that the conditional measure below is the average of the conditional measures above. Let $\phi:[0,1]\rightarrow[0,\infty),\ \phi(t)=-t\log t$, which is positive and has the second derivative negative on $(0,1)$. Then, by Jensen inequality, for $\mu$ almost every $x$ we have
\begin{eqnarray*}
\phi(\nu_{\pi(x)}(A))&=&\phi\left(\int\mu_z(\pi^{-1}(A))d\overline\mu_x(z)\right)\\
&\geq&\int\phi\left(\mu_z(\pi^{-1}(A))\right)d\overline\mu_x(z),
\end{eqnarray*}
and the equality holds if and only if $\mu_z(\pi^{-1}(A))=\nu_{\pi(x)}(A)$ for $\overline\mu_x$ almost every $z\in\pi^{-1}(x)$. This implies that
\begin{eqnarray*}
H_{\nu_{\pi(x)}}\left(\overline f^{-1}\overline\xi\mid_{\overline\xi(\pi(x))}\right)&=&\sum_{A\in\overline f^{-1}\overline\xi,A\subset\overline\xi(\pi(x))}\phi\left(\nu_{\pi(x)}(A)\right)\\
&\geq&\int\sum_{B\in f^{-1}\xi,B\subset\xi(z)}\phi(\mu_z(B))d\overline\mu_x(z)\\
&=&\int H_{\mu_z}\left(f^{-1}\xi\mid_{\xi(z)}\right)d\overline\mu_x(z).
\end{eqnarray*}
This says that the entropy on a piece of the partition below is greater or equal than the average of the entropies on the corresponding pieces of the partition above. Now since $\pi$ preserves the quotient measures of $\pi^{-1}(\overline\xi)$ for $\mu$ and $\overline\xi$ for $\nu$, by integrating the inequality above we get the desired inequality. One can see that the equality holds if and only if for $\mu$ almost every $x$, for any $A\in f^{-1}\xi$ with $A\subset\xi(x)$, we have $\mu_x(A)=\nu_{\pi(x)}(\pi(A))$, i.e. $\pi_*\mu_x$ and $\nu_{\pi(x)}$ agree on the $\sigma$-algebra generated by $\overline f^{-1}\overline\xi$. A similar argument shows that the statement must hold also for the $\sigma$-algebra generated by $\overline f^{-n}\overline\xi$ for any positive integer $n$, and since the limit at infinity is the partition into points, we obtain that $\pi_*\mu_x=\nu_{\pi(x)}$ for $\mu$ almost every $x$.

\end{proof}

Now let us explain how we use this result in our setting. Let $L$ be a linear Anosov diffeomorphism over the torus $\TT^d$, irreducible and with simple real eigenvalues with distinct absolute values, $N$ be a compact Riemannian manifold without boundary, $M= \TT^d\times N$, and $f_0: M \to M$ be a $C^1$ partially hyperbolic skew product diffeomorphism
$$
f_0((x,y))=(L(x), h_x(y))
$$
with $\{\cdot\}\times N$ corresponding to the center bundle.

Like in the Subsection \ref{standhyp}, denote by $\lambda^s_{l}<\lambda^s_{l-1}<\cdots<\lambda^s_1<0<\lambda^u<\cdots<\lambda^u_k$ the exponents of $L$ and by $E^s_{l}\oplus E^s_{l-1}\oplus \cdots\oplus E^s_1\oplus E^u_1\oplus \cdots\oplus E^u_k$ the invariant splitting corresponding to the eigenvectors. Let $\cF^s_i, \cF^u_j$ be the corresponding one dimensional foliations, and $\cF^s_{1,i}$, $\cF^u_{1,j}$ the (weak unstable or 'center') foliations tangent to $\oplus_{t=1}^iE^{s}_t$ and respectively $\oplus_{t=1}^jE^{u}_t$, for all $1\leq i\leq l$, $1\leq j\leq k$. 

The following result can be deduced easily from \cite{HPS77}.

\begin{lemma}\label{l.perturbationquotient}

For any $C^1$ small perturbation $f$ of $f_0$, $f$ is dynamically coherent, the center foliation $\cF^c_f$ of $f$ forms a fiber bundle, and each fiber is homeomorphic to $N$. There exists a projection $\pi:\TT^d\times N\rightarrow\TT^d$ which semiconjugates $f$ with $L$, and which takes the (compact) center leaves of $f$ to points.

Moreover, the splittings of the stable and unstable bundles into one-dimensional sub-bundles persist: we have $E^{s,f}=E^{s,f}_{l}\oplus \cdots\oplus E^{s,f}_1$ and $E^{u,f}= E^{u,f}_1\oplus \cdots\oplus E^{u,f}_k$. The bundles $E^{s,f}_i$ and $E^{u,f}_j$, $\oplus_{t=1}^iE^{s,f}_t$ and $\oplus_{t=1}^jE^{u,f}_t$, $\oplus_{t=1}^iE^{s,f}_t\oplus E^{c,f}$ and $\oplus_{t=1}^jE^{u,f}_t\oplus E^{c,f}$ integrate to foliations $\cF^{s,f}_i$ and $\cF^{u,f}_j$, $\cF^{s,f}_{1,i}$ and $\cF^{u,f}_{1,j}$, $\cF^{cs,f}_{1,i}$ and $\cF^{cu,f}_{1,j}$, for all $1\leq i\leq l$, $1\leq j\leq k$. The weak stable and unstable leaves of $f$, $\cF^{s,f}_{1,i}$ and $\cF^{u,f}_{1,j}$ project homeomorphically by $\pi$ to the weak stable and unstable leaves of $L$, $\cF^{s}_{1,i}$ and $\cF^{u}_{1,j}$.

\end{lemma}

Consequently, in order to apply Theorem \ref{p.partialentropy} we will consider  $M=\TT^d\times N$, $\overline M=\TT^d$, $\overline f=L$. The measure $\mu$ will be the volume on $M$, and $\nu=\pi_*\mu$. The expanding foliations $\cF$ and $\overline\cF$ will be conveniently chosen intermediate stable and unstable foliations of $f$ and $L$.

We also have the following result:

\begin{proposition}\label{p.base}

Let $\nu$ be any ergodic measure of $L$, then
$$
h_\nu(L)=h_\nu(L,\cF^u)\leq \sum_{i=1}^k \lambda^u_i=\lambda_L^u,
$$
and the equality holds if and only if $\nu$ is the volume, the unique measure of maximal entropy of $L$.

\end{proposition}

\begin{proof}

By the entropy formula for Ledrappier-Young \cite{LY1}, $h_\nu(L)=h_\nu(L,\cF^u)$.

It is well-known that the topological entropy of $L$ equals to $ \lambda^u_L$, hence by variation principle, $h_\nu(L)\leq \lambda^u_L$. Moreover, the equality holds if and only if $\nu$ is the unique measure of maximal entropy of $L$.

\end{proof}

\section{Proof of Theorem~\ref{main.cneterfible}}\label{s.skew}

In this section we will give the proof of Theorem~\ref{main.cneterfible}. Throughout this section, let $f\in \Diff^2(M)$ be a partially hyperbolic diffeomorphism with one dimensional stable and unstable foliation, satisfying the hypothesis of Theorem \ref{main.cneterfible}. In particular we assume that it satisfies the conclusion of Lemma \ref{l.perturbationquotient}, so it fibers over the Anosov automorphism $L$ of the 2 torus.

\begin{proof}[Proof (Theorem \ref{main.cneterfible})]
We separate the proof in three parts.

\subsection{The center foliation is $C^{1+\vep}$}

Because $f$ is volume preserving, the Lebesgue measure is a Gibbs $u$-state for $f$. By the entropy formula of Ledrappier-Young \cite{LY2} and our hypothesis,

\begin{equation}\label{eq.partialentropyabove}
h_{\Leb}(f,\cF^{u,f})=\int \lambda^u(f,x)d\Leb(x)=\lambda^u_L.
\end{equation}

Denote by $\nu=\pi_*(\Leb)$.

\begin{lemma}\label{l.suinvariantandbase}

The Lebesgue measure is $c$-invariant for $\cF^{u,f}$ and $\cF^{s,f}$, and $\nu$ is the Lebesgue measure on $\TT^2$, or the unique maximal entropy measure of $L$.

\end{lemma}

\begin{proof}

By Theorem~\ref{p.partialentropy} applied to $\cF^{u,f}$ and $\cF^u$, we have $h_{\Leb}(f,\cF^{u,f})\leq h_\nu(L,\cF^u)$. Combine with Proposition~\ref{p.base} and \eqref{eq.partialentropyabove}, we have
$$
\lambda^u_L=h_{\Leb}(f,\cF^{u,f})\leq h_\nu(L,\cF^u)= h_\nu(L)\leq \lambda^u_L,
$$
and so the above items must be all equal. Thus $h_\nu(L,\cF^{u,f})= h_\nu(L)= \lambda^u_L$, by Proposition~\ref{p.base}, $\nu$ is the volume, or the unique measure of maximal entropy of $L$.

Moreover, we have shown that $h_{\Leb}(f,\cF^{u,f})= h_\nu(L,\cF^u)$, then by Theorem~\ref{p.partialentropy}, the Lebesgue measure is $c$ invariant for $\cF^{u,f}$. A similar argument works for $\cF^{s,f}$ and $f^{-1}$.

\end{proof}

Let $\overline{\xi}^u$ be a partition subordinated to $\cF^u$, constructed from a Markov partition of $L$, by intersecting pieces of the Markov partition with local unstable leaves (it is easy to check that such a partition is subordinated to $\cF^u$). We can lift it to a subordinated partition $\xi^u$ of $\cF^{u,f}$ by taking the preimage by $\pi$ and intersecting with local unstable leaves. Let $\mu_x$ be the conditional measures of the volume on $M$ along $\xi^u$ and $\nu_y$ the conditional measures of the volume on $\TT^2$ along $\overline{\xi}^u$. From the definition of $c$ invariance we know that, for $\mu$ almost every $x\in M$, we have that
\begin{equation}\label{eq.projmeasure}
\pi_{*}\mu_x=\nu_{\pi(x)}.
\end{equation}

Let $B$ be any piece of the Markov partition. Note that $\nu_{\pi(x)}$ is the normalized Lebesgue measure on the unstable segment of $W^u$, while $\mu_x$ is absolutely continuous with respect to Lebesgue on the $\cF^u$, and has the density uniformly continuous on $\pi^{-1}(B)$ and uniformly $C^1$ along the leaves. As we observed before, even if initially $\mu_x$ is defined only almost everywhere, the uniform continuity of the density and the full support of the volume allow us to define it for {\it every} $x\in \pi^{-1}(B)$, and the family $\mu_x$ will be continuous in $x$. This in turn implies that the above relation \eqref{eq.projmeasure} holds for {\it every} $x\in \pi^{-1}(B)$.

Now, repeating the argument from Theorem \ref{main.technique}, since $\pi$ restricted to an unstable piece of $f$ takes a smooth measure with $C^1$ density to the Lebesgue measure on a segment of $W^u$, and we are in the one dimensional situation, we obtain that $\pi\mid_{\xi^u(x)}$ is a $C^2$ diffeomorphism with $\overline{\xi}^u(\pi(x))$, uniformly with respect to $x\in B$. Since the central holonomy between two unstable leaves of $f$, inside a center-unstable leaf of $f$, is the composition of two diffeomorphisms as above, we get that the center holonomy restricted to the center-unstable foliation is also uniformly $C^2$ inside $\pi^{-1}(B)$. Now, since $M$ is covered by finitely many pieces of the Markov partition, we obtain that the center holonomy restricted to the center-unstable foliation is uniformly $C^2$ on the whole manifold $M$ (one can eventually modify the Markov partition in order to deal with the boundary points).

Because for a $C^2$ partially hyperbolic diffeomorphism, the center bundle is always Holder, we know that the center leaves are uniformly $C^{1+\vep}$ for some $\vep>0$. By Journ\'{e} regularity lemma for foliations (Lemma~\ref{l.Jourfoliation}), we obtain that $\cF^{c,f}\mid_{\cF^{cu,f}_{loc}(x)}$ is uniformly $C^{1+\vep}$.

In a similar way, one can show that $\cF^{c,f}\mid_{\cF^{cs,f}_{loc}(x)}$ is uniformly $C^{1+\vep}$. We will use the following version of Journ\'{e} regularity lemma for foliations (Lemma 4.5 of \cite{DX}).

\begin{lemma}\label{l.journe3}
Suppose $\cF_i,\ i = 1, 2, 3$ are foliations of a smooth manifold M with uniformly $C^{r+}$ leaves. We assume $\cF_{1,2}$ subfoliate $\cF_3$ and $\cF_1$ transverse to $\cF_2$ within $\cF_3$. Moreover $\mathcal W := \cF_1 \cap \cF_2$ is a $C^{r+}$ foliation within $\cF_i,\ i = 1, 2$ respectively, then $\mathcal W$ is a $C^{r+}$ foliation in $\cF_3$.
\end{lemma}

Applying the above result for $r=1+\vep$, $\mathcal W=\cF^{c,f}$, $\cF_1=\cF^{cu,f}$, $\cF_2=\cF^{cs,f}$, and $\cF_3=M$, we will obtain that $\cF^{c,f}$ is a $C^{1+\vep}$ foliation. This concludes the first part of the proof.

\subsection{If $\dim(E^c)=1$ then the center foliation is $C^{\infty}$}
For the following lemma, we assume that $f$ is $C^\infty$, the center bundle is one dimensional, and the center leaf is homeomorphic to the one dimensional circle.

\begin{lemma}\label{l.centersmooth}

The center foliation $\cF^{c,f}$ is $C^\infty$.

\end{lemma}

\begin{proof}

As we mentioned in the first part of the proof of Theorem \ref{main.cneterfible}, $\pi\mid_{\xi^u(x)}$ is a $C^2$ diffeomorphism with $\overline{\xi}^u(\pi(x))$, uniformly with respect to $x$. Since we assume now that $f$ is $C^{\infty}$, it follows that the leaves of $\cF^u$ and $\cF^{u,f}$ are uniformly $C^{\infty}$, the disintegrations of the Lebesgue measures along $\xi^u(x)$ and $\overline{\xi}^u(\pi(x))$ have uniform $C^{\infty}$ densities, so the argument from the previous section shows in fact that the center holonomy restricted to the center-unstable foliation is uniformly $C^{\infty}$ inside $\pi^{-1}(B)$. Consequently, in order to apply again Lemma \ref{l.journe3} and show that the center foliation is $C^\infty$, it suffices to prove that the center leaves $\cF^{c,f}(x)$ are $C^\infty$ for every $x\in M$ (classical results from \cite{HPS77} will imply that $\cF^{cs,f}$ and $\cF^{cu,f}$ also have $C^{\infty}$ leaves).

We will need to use the notion of center bunching. Given $r>0$, we say that the partially hyperbolic diffeomorphism $f$ is \emph{$r$-bunched} if there exists $k\geq 1$ such that for any $p\in M$ we have:

\begin{equation}\label{eq.centerbunch}
\begin{aligned}
&\sup \|D_pf^k\mid_{E^{s,f}}\|\|(D_pf^k\mid_{E^{c,f}})^{-1}\|\|D_pf^k\mid_{E^{c,f}}\|^r<1&\\
&\sup \|(D_pf^k\mid_{E^{u,f}})^{-1}\|\|D_pf^k\mid E^{c,f}\|\|(D_pf^k\mid_{E^{c,f}})^{-1}\|^r<1.&\\
\end{aligned}
\end{equation}

By \cite{PSW}, the center leaf of $f$ is $C^r$ if $f$ is $r$-bunched. Then we only need to verify that $f$ is $\infty$-bunched. Since the center bundle is one dimensional, we only need to verify that for any $r>1$, there exists some $k>0$ such that:
$$
\sup_p \|D_pf^k\mid_{E^s}\|\|D_pf^k\mid_{E^c}\|^r<1 \text{ and } \sup_p \|(D_pf^k\mid_{E^u})^{-1}\|\|D_pf^k\mid_{E^c}\|^{-r}<1.
$$
These inequalities are clearly satisfied if $Df\mid_{E^{c,f}}$ is neutral, in the sense that there exists $K>0$ such that, for any $n>0$, we have
$$
\frac{1}{K}\leq \|D_p f^n\mid_{E^{c,f}}\|\leq K.
$$

Because $\cF^{c,f}$ is a $C^1$ foliation, by Fubini theorem, there exists a family of continuous disintegration $\{\Leb^c_{x}\}_{x\in M}$ of the Lebesgue measure along the center leaves, such that on each center leaf, the disintegration is equivalent to the Lebesgue measure of the leaf with continuous density, and the density is bounded uniformly from above and away from zero. By the uniqueness of the disintegration, this disintegration is invariant by $f$, meaning that:
\begin{equation}\label{eq.iterationinv}
f_*(\Leb^c_x)=\Leb^c_{f(x)} \text{ for any }x\in M.
\end{equation}
Write $d\Leb^c_{x}(y)=\phi(y)d\Leb\mid_{\cF^{c,f}(x)}$, then there exists $K_1>0$ such that, for any $y\in M$, we have
\begin{equation}\label{eq.uniformdensity}
\frac{1}{K_1}<\phi(y)<K_1.
\end{equation}
Then we can finish the proof applying the following lemma:

\begin{lemma}
$$
\frac{1}{K_1^2}\leq \|D_p f^n\mid_{E^{c,f}}\|\leq K_1^2.
$$

\end{lemma}

\begin{proof}

Suppose there exists $x\in M$ and $n>0$ such that $\|D_x f^n\mid_{E^{c,f}}\|> K_1^2$. Then there exists a small center segment containing $x$ such that $|f^n(I)|>K_1^2 |I|$, where $|I|$ denotes the length of segment. By \eqref{eq.uniformdensity}, $\Leb^c_x(I)\leq K_1|I|$ and $\Leb^c_{f^n(x)}(f^n(I))\geq \frac{1}{K_1}|f^n(I)|> K_1 |I_1|$.
Thus $\Leb^c_x(I)\neq \Leb^c_{f^n(x)}(f^n(I))$, a contradiction with \eqref{eq.iterationinv}.

\end{proof}
\end{proof}

\subsection{Conjugacy with the true skew-product}
Now we will explain how to (smoothly) conjugate $f$ with a true skew product over $L$. Choose $y_0\in N$ and consider the torus $\TT^2_0=\TT^2\times \{y_0\}$. Because for $f_0$ the center foliation is transverse to $\TT^2\times\{y\}$, by the continuity of center foliation with respect to the diffeomorphisms, the center foliation of $f$ is also transverse to $\TT^2\times\{y\}$, for all $y\in N$. Thus for every $y,z\in N$ we have the center holonomy $\cH^c_{y,z}:\TT^2\times\{y\}\rightarrow\TT^2\times\{z\}$, which is a $C^{1+\vep}$ diffeomorphism. We may identify $\TT^2_0=\TT^2$, and define $\pi_c:M\to\TT^2_0$ the projection along the center foliation of $f$. Since the center foliation of $f$ is $C^{1+\vep}$, thus $f$ induces a $C^{1+\vep}$ quotient map $\overline f: \TT^2\to \TT^2$, $\overline f= \pi_c \circ f\mid_{\TT^2_0}$. Moreover, $\overline f$ preserves a smooth volume measure $\nu=\pi_*(Leb\mid M)$ of $\TT^2_0$.

\begin{lemma}\label{l.2d}

$\overline f$ is an Anosov diffeomorphism homotopic to $L$.

\end{lemma}

\begin{proof}

$\overline f$ is homotopic to $L$ because it is $C^0$ close to $L$.

For any $y\in M$, denote by $y'=\pi_c(y)\in \TT^2_0$, and let $E^{u,\overline f}(y')=D\pi_c(E^{u,f}(y))$. It is clear that the definition is independent of $y\in\pi_c^{-1}(y')$ because $\pi_c$ preserves $\cF^{cu,f}$ (in fact an equivalent definition could be $E^{u,\overline f}(y')=E^{cu,f}(y')\cap \left(T\TT^2\times\{0\}\right)$).

Because the bundle $E^{u,f}$ is uniformly away from the bundle $E^{c,f}=Ker(D\pi_c)$, there exists $K_3>0$ such that
\begin{equation}\label{eq.boundproj}
\frac{1}{K_3}<m(D\pi_c\mid_{E^{u,f}(y)})\leq \|D\pi_c\mid_{E^{u,f}(y)}\|<K_3,
\end{equation}
where $m(L)$ denotes the minimal norm of the linear map $L$.

We have that
\begin{equation}\label{eq.leafconjugate}
\overline f^n=\pi_c\circ f^n.
\end{equation}

Using  \eqref{eq.boundproj} and \eqref{eq.leafconjugate}, and the uniform expansion of $Df$ along $\cF^{u,f}$, we obtain that there exists $m_0>0$ such that
$$
m(D\overline f^{m_0}\mid_{E^{u,\overline f}_{y'}})>1 \text{ for any  } y'\in \TT^2_0.
$$
Similarly, we can show that
$$
\|D\overline f^{m_0}\mid_{E^{s,\overline f}_{y'}}\|<1 \text{ for any  } y'\in \TT^2_0.
$$
The two inequalities above also show that $E^{u,\overline f}\oplus E^{s,\overline f}$ is a hyperbolic splitting, thus $\overline f$ is an Anosov diffeomorphism.

\end{proof}

Let $h_1:\TT^2\times N\rightarrow\TT^2\times N$, $h_1(x,y)=(\cH^c_{y,y_0}(x,y),y)$. Then clearly $h_1$ is a $C^{1+\vep}$ diffeomorphism which takes the center leaves $W^{c,f}(x,y)$ to $\{\pi_c(x,y)\}\times N$. Then $h_1$ smoothly conjugates $f$ to $f_1=h_1\circ f\circ h_1^{-1}$ which is a $C^{1+\vep}$ true skew product over $\overline f|_{\TT^2_0}$. The average Lyapunov exponents of $f$ and $f_1$ must coincide since they are smoothly conjugated, and the stable and unstable Lyapunov exponents of $f_1$ must coincide with the Lyapunov exponents of $\overline f$ since $f_1$ is a true skew product over $\overline f$. Then by our hypothesis $\overline f$ and $L$ must have the same exponents, so they must be $C^{1+\vep}$ conjugated by a diffeomorphism $h'$. Taking $h_2=(h',Id\mid_N)$ and $h=h_2\circ h_1$ we obtain that $h$ is a $C^{1+\vep}$ conjugacy between $f$ and a true skew product over $L$.

Now let us assume that $N=S^1$ and the center foliation of $f$ is $C^{\infty}$. We know that $f$ is $C^\infty$ conjugated with a true skew product map over $L$: $f_2: \TT^2\times S^1\to \TT^2\times S^1$, $f_2(x,y)=(L(x), g_x(y))$. We also know that $f_2$ preserves a $C^{\infty}$ volume $\nu$. Consider the continuous family of disintegration $\{\nu^c_x\}_{x\in M}$ of $\nu$ along the center foliation of $f_2$. Because the center foliation is $C^\infty$, and $\nu$ is $C^{\infty}$ equivalent to Lebesgue, we know by Fubini theorem, that for every $x\in M$, $\nu^c_x$ is equivalent to $\Leb\mid_{\cF^c(x)}$ with $C^\infty$ density. Moreover, the disintegration is invariant under the iterations of $f_2$.

For any two points $y,z$ in the same center leaf, we use $[y,z]$ to denote the center segment from $y$ to $z$ in the anti-clockwise direction. Let $h_3:M\rightarrow M$, $h_3(x,y)=(x, \nu^c_x([0,y]))$. It is easy to see that $h_3$ is a $C^{\infty}$ diffeomorphism and $h_{3*}\nu=\Leb\mid_M$. Then $h_3$ is a $C^{\infty}$ conjugacy between $f_2$ and $f'$, which is a $C^{\infty}$ true skew product over $L$ which preserves the Lebesgue measure on the center circles. But this implies that the restrictions of $f'$ to the center fibers must be rotations which finishes our proof.

\end{proof}

\section{Proof of Theorem~\ref{main.higherbase}}\label{higherbase}

The proof of the Theorem~\ref{main.higherbase} is parallel to the proof of Theorem~\ref{main.cneterfible} and the proof of Theorem \ref{main.higherdimension}.

\begin{proof}

Let $\nu:=\pi_*\Leb$ be an invariant measure for $L$ on $\TT^d$.

Observe that by Lemma \ref{l.perturbationquotient} we know that the unstable foliation of $f$ decomposes into one dimensional unstable foliations $\cF^{u,f}_i$, which have uniformly $C^{1+\vep}$ leaves, since all the sub-bundles are H\H older continuous. We know by hypothesis that $h_{\Leb}(f,\cF^{u,f}_i)=\lambda^u_i$, for all $1\leq i\leq k$. By Ruelle inequality, the corresponding Lyapunov exponent of $f$ for $\Leb$ and $\cF^{u,f}_i$ must satisfy
$$
h_{\Leb}(f,\cF^{u,f}_i)=\lambda^u_i\leq \lambda^u_i(f,\Leb),\ \ \forall 1\leq i\leq k.
$$
If one of these inequalities is strict, then taking the sum for all $i$ we obtain
\begin{equation}\label{eq.exponentssum}
\lambda^u(f,\Leb)=\sum_{i=1}^k\lambda^u_i(f,\Leb)>\sum_{i=1}^k\lambda^u_i=\lambda^u(L),
\end{equation}
which is a contradiction, because the absolute continuity of $\cF^{u,f}$ and the so-called Ma\~ ne argument would imply that the volume of the unstable leaves grows faster than the topological growth of the linear part, which is impossible (see \cite{SX09} or \cite{Sa}).

An alternative proof of the contradiction is the following: By Pesin formula we have that $h_{\Leb}(f,\cF^{u,f})=\lambda^u(f,\Leb)$; by the Ruelle inequality we have $h_{\nu}(L,\cF^u)\leq\lambda^u(L)$; these facts together with \eqref{eq.exponentssum} give $h_{\Leb}(f,\cF^{u,f})>h_{\nu}(L,\cF^u)$, but this contradicts Theorem \ref{p.partialentropy}.

Consequently we have that $\lambda^u_i(f,\Leb)= h_{\Leb}(f,\cF^{u,f}_i)=\lambda^u_i,\ \ \forall 1\leq i\leq k$. Using the converse of the Pesin formula for expanding foliations, we obtain that the volume must be an expanding Gibbs state for each foliation $\cF^{u,f}_i$, $1\leq i\leq k$. The same argument works for the corresponding one dimensional stable foliations and for $f^{-1}$.

We also have that $\lambda^u(f,\Leb)=\lambda^u(L)$, and by Lemma \ref{l.suinvariantandbase} we know that $\nu$ must be the volume on $\TT^d$ and the Lebesgue measure is $c$-invariant for $\cF^{u,f}$ and $\cF^{s,f}$.

Next, we will show by induction that the center holonomy is uniformly $C^{1+\vep}$ between the leaves of $\cF^{u,f}_i$, for all $1\leq i\leq k$.

For $i=1$, we know from Lemma \ref{l.perturbationquotient} that $\cF^{u,f}_1$ is invariant by the center holonomy, and it projects by $\pi$ to the invariant foliation $\cF^u_1$ of $L$. Furthermore the volume above projects to the volume below, and the conditional entropies along the expanding foliations above and below coincide. In conclusion, we can apply Theorem \ref{p.partialentropy} and the same arguments from the proof of Theorem \ref{main.cneterfible} in order to conclude that the center holonomy is uniformly $C^{1+\vep}$ along $\cF^{u,f}_1$ leaves inside $\cF^{cu,f}_1$, and the projection $\pi$ from $\cF^{u,f}_1$ leaves of $f$ to $\cF^{u}_1$ leaves of $L$ is uniformly $C^{1+\vep}$ diffeomorphism.

Now assume that the projection $\pi$ is a uniform $C^{1+\vep}$ diffeomorphism from $\cF^{u,f}_{j}$ leaves of $f$ to $\cF^{u}_{j}$ leaves of $L$ and that the center holonomy is uniformly $C^{1+\vep}$ along $\cF^{u,f}_{j}$ leaves inside $\cF^{cu,f}_{j}$, for all $1\leq j\leq i$. We have an analog of Gogolev's result from Lemma \ref{l.inductioncenterleaf} which we can apply in this situation.

\begin{lemma}\label{l.gogolevskew}
Suppose that the projection $\pi$ is a uniform $C^{1+\vep}$ diffeomorphism from $\cF^{u,f}_{i}$ leaves of $f$ to $\cF^{u}_{i}$ leaves of $L$. Suppose also that the leaves of $\cF^{u,f}_{i,k}$ project by $\pi$ homeomorphically to leaves of $\cF^u_{i,k}$. Then the leaves of $\cF^{u,f}_{i+1,k}$ project by $\pi$ homeomorphically to leaves of $\cF^u_{i+1,k}$.
\end{lemma}

\begin{proof}
The proof follows closely the one from Gogolev, we just replace the conjugacy with the semiconjugacy, so we will give just a sketch. Suppose that the conclusion is not true, then there exists $x\in\TT^d\times N$, $\overline x=\pi(x)\in\TT^d$, $y\in\cF^{u,f}_{i+1,k}(x)$, with $\overline y=\pi(y)\notin\cF^u_{i+1,k}(\overline x)$. Since $\pi$ takes $\cF^{u,f}_{i,k}$ to $\cF^u_{i,k}$, we have that $\overline y\in\cF^u_{i,k}(\overline x)$. Let $\overline z=\cF^u_{i+1,k}(\overline x)\cap\cF^u_{i}(\overline y)$.

Let $\cH_{x,y}$ be the holonomy along $\cF^{u,f}_{i+1,k}$, from $\cF^{u,f}_{i}(x)$ to $\cF^{u,f}_{i}(y)$. Since $E^{u,f}_{i+1,k}$ is $C^1$ inside $\cF^{u,f}_{i,k}$, we have that $\cH_{x,y}$ must be (uniformly) $C^1$. Let
$$
\overline\cH_{\overline x,\overline y}=\pi\mid_{\cF^{u,f}_{i}(y)}\circ\cH_{x,y}\circ\left(\pi\mid_{\cF^{u,f}_{i}(x)}\right)^{-1}.
$$
As a composition of $C^1$ functions, $\overline\cH_{\overline x,\overline y}$ must be $C^1$. Observe also that
$$
\overline\cH_{\overline x,\overline y}=L^n\circ\overline\cH_{L^{-n}\overline x,L^{-n}\overline y}\circ L^{-n},
$$
so
$$
D\overline\cH_{\overline x,\overline y}=e^{\lambda^u_i}D\overline\cH_{L^{-n}\overline x,L^{-n}\overline y}e^{-\lambda^u_i}=D\overline\cH_{L^{-n}\overline x,L^{-n}\overline y}.
$$
Also
$$
D\overline\cH_{L^{-n}\overline x,L^{-n}\overline y}=\left(D\pi\mid_{\cF^{u,f}_{i}(f^{-n}(y))}\right) D\cH_{f^{-n}(x),f^{-n}(y)} \left(D\pi\mid_{\cF^{u,f}_{i}(f^{-n}(x))}\right)^{-1}.
$$
Since $d(f^{-n}(x),f^{-n}(y))\rightarrow 0$, uniform continuity of $D\pi\mid_{\cF^{u,f}_i}$ and the fact that $\cH$ is uniformly $C^1$ (given by a uniform $C^1$ bundle inside $\cF^{u,f}_{i,k}$), one gets that $D\overline\cH_{\overline x,\overline y}=1$, so $\overline\cH_{\overline x,\overline y}$ must be a translation by $d(\overline y,\overline z)$. The minimality of $\cF^u_i$ and the continuity of $\pi$ would then imply that the lift of $\pi(\cF^{u,f}_{i+1,k})$ to the universal cover $\mathbb R^d$ of $\TT^d$ contains all the points $\overline x+n(\overline y-\overline x)$, so must lie in a band of slope $(\overline y-\overline x)$, but this would contradict the invariance under $L$.

\end{proof}

We apply inductively Lemma \ref{l.gogolevskew}, together with the observation that if $\pi$ projects $\cF^{u,f}_{i+1,k}$ to $\cF^{u}_{i+1,k}$, then it will also project $\cF^{u,f}_{i+1}$ to $\cF^{u}_{i+1}$ (because we know from \ref{l.perturbationquotient} that $\pi$ projects $\cF^{u,f}_{1,i+1}$ to $\cF^{u}_{1,i+1}$).

Once we establish that the projection $\pi$ sends homeomorphically the leaves of $\cF^{u,f}_{i+1}$ to leaves of $\cF^{u}_{i+1}$, we can apply Theorem \ref{p.partialentropy} and the same methods as above in order to conclude that the projection $\pi$ is a uniform $C^{1+\vep}$ diffeomorphism from $\cF^{u,f}_{i+1}$ leaves of $f$ to $\cF^{u}_{i+1}$ leaves of $L$ and that the center holonomy is uniformly $C^{1+\vep}$ along $\cF^{u,f}_{i+1}$ leaves inside $\cF^{cu,f}_{i+1}$.

Now, in order to finish the proof of Theorem \ref{main.higherbase}, we apply as before the Journ\'e regularity result for foliations, and we use the same arguments for $f^{-1}$ in order to deal with the stable direction. The proof of the conjugacy with the true skew product is similar to the proof in Theorem \ref{main.cneterfible}.

\end{proof}

\section{Further remarks}

In this last subsection we want to make some comments about (possible) extensions of our results.

\begin{enumerate}

\item {\bf Anosov diffeomorphisms with non-simple spectrum.}

It seems possible to extend Theorem \ref{main.higherdimension} to Anosov maps without simple spectrum, under some additional conditions. One may hope to use the work of de la Llave, Gogolev, Kalinin and Sadovskaya on the rigidity of these maps. Very recently Gogolev-Kalinin-Sadovskaya obtained results in this direction in \cite{GKS2}.

\item {\bf Derived from Anosov maps.}

We also expect that some of the results from this paper can be extended to derived from Anosov maps. We propose the following conjecture.

\begin{conjecture}
Let $f$ be a smooth volume preserving, derived from Anosov partially hyperbolic diffeomorphism on $\TT^3$, isotopic to the Anosov automorphism $L$. Assume that the three (average) Lyapunov exponents of $f$ equal to the three Lyapunov exponents of $L$. Then $f$ is smoothly conjugated to $L$.
\end{conjecture}

Let us remark here that our methods from this paper do not extend immediately to the DA situation, because in this case the center foliation is not uniformly expanding.

\item {\bf Skew products over higher dimensional Anosov maps.}

It may be possible to extend Theorem \ref{main.higherbase} in order to obtain the following result:

\begin{conjecture}Let $L$ be an irreducible Anosov automorphism of $\TT^d$ with simple real eigenvalues with distinct absolute values, $N$ a compact manifold, $M=\TT^d\times N$, and $f_0:M\to M$ a partially hyperbolic skew product over $L$. Let $f$ be a $C^2$ volume preserving  partially hyperbolic diffeomorphism $C^1$ close to $f_0$, and assume that the (strong) stable and unstable exponents of $f$ coincide with the exponents of $L$. Then the center foliation of $f$ is $C^{1+\vep}$ for some $\vep>0$, and $f$ is $C^{1+\vep}$ conjugated with a (true) skew product over $L$.
\end{conjecture}

\item{\bf Time one maps of hyperbolic flows.}

In general one cannot expect that the perturbation of the time one map of a hyperbolic flow can be embedded into another hyperbolic flow, even if the perturbation is volume preserving and it preserves the Lyapunov exponents. One can construct easily counterexamples even in dimension 3, for any hyperbolic flow, by increasing slightly the speed of the flow and mixing locally the stable and center directions, and respectively the unstable and center directions.

However it seems that some results could be obtained for geodesic flows on manifolds of constant negative curvature, assuming that the unstable (respectively the stable) exponents are all equal, and the dimension is larger than 2 (see for example \cite{Bu}, \cite{BX}).

\end{enumerate}

\end{document}